\documentclass{article}
\usepackage{macros}
\usepackage{xy}
\usepackage{comment}
\xyoption{all}

\newcommand\subfrac[2]{\genfrac{}{}{0pt}{}{#1}{#2}}

\makeatletter
\newcommand{\subjclass}[2][1991]{%
  \let\@oldtitle\@title%
  \gdef\@title{\@oldtitle\footnotetext{#1 \emph{Mathematics subject classification.} #2}}%
}
\newcommand{\keywords}[1]{%
  \let\@@oldtitle\@title%
  \gdef\@title{\@@oldtitle\footnotetext{\emph{Key words and phrases.} #1.}}%
}
\makeatother

\title{Intersections of Hecke correspondences on the modular varieties of $\mathcal{D}$-elliptic sheaves}
\author{\"Ozge \"Ulkem \footnote{
		{Institute of Mathematics, Academia Sinica, Astronomy Mathematics Building, No. 1,
Roosevelt Rd. Sec. 4, Taipei, Taiwan, 10617}} \and 
 Fu-Tsun Wei \footnote{Department of Mathematics, National Tsing Hua University and National Center for Theoretical Sciences, Hsinchu City 30042, Taiwan R.O.C}
 }
\date{}

\begin{document}

\subjclass[2020]{Primary 11R58, 11G18}

\keywords{$\cD$-elliptic Sheaf, Hecke Correspondence, Intersection, Class Number Relation}

\maketitle

\begin{abstract}
This paper studies the intersections of Hecke correspondences on the modular varieties of \( \cD \)-elliptic sheaves in the higher-rank setting, where \( \cD \) is a ``maximal order'' in a central division algebra \( D \) over a global function field \( k \).  
Assuming that \( \dim_k(D) = r^2 \), where \( r \) is a prime distinct from the characteristic of \( k \), we express the intersection numbers of Hecke correspondences as suitable combinations of modified Hurwitz class numbers of ``imaginary orders''. This result establishes a higher-rank analogue of the classical class number relation.

\end{abstract}

\section{Introduction}

\indent Let $X$ denote the modular curve of full level. For each non-square $n \in \mathbb{N}$, the Hecke correspondence $T_n$ on $X$ corresponds to a $1$-cycle $\mathcal{Z}_n$ on $X \times X$. The geometric version of the celebrated class number relation (initiated by Kronecker \cite{Kron} and extended by Hurwitz \cite{Hur}) expresses the ``finite part'' of the intersection number $i(\mathcal{Z}_1 \cdot \mathcal{Z}_n)$ (i.e.\ excluding the ``cuspidal intersection'') as 
\[
i(\mathcal{Z}_1 \cdot \mathcal{Z}_n)_{\rm fin} = \sum_{\subfrac{t \in \mathbb{Z}}{t^2 < 4n}} H(t^2 - 4n) \quad \big(= \sum_{\subfrac{m \in \mathbb{N}}{m \mid n}} \max(m, n/m)\big).
\]
Here, $H(N)$ denotes the ``Hurwitz class number'' for $N \in \mathbb{N}$. The fact that these intersection points are actually ``CM'' provides a conceptual foundation for the proof of this relation, as it connects the intersection in question to Eichler's theory of optimal embeddings of imaginary quadratic orders into matrix rings. This phenomenon is also evident in the context of Hilbert modular surfaces associated with real quadratic fields, as demonstrated in the celebrated work of Hirzebruch and Zagier \cite{H-W}. Since then, the arithmetic significance of intersection numbers between special cycles on modular varieties has been extensively studied (see \cite{K-M} and \cite{K-MI}), leading to the so-called ``geometric Siegel-Weil formula'' with profound applications (see \cite{Kud1}, \cite{Cog} \cite{K-MII}, \cite{Kud2}, \cite{Fun}).

The primary aim of this paper is to explore this phenomenon within the function field setting and to describe the intersection numbers of ``Hecke correspondences'' on the modular varieties of $\cD$-elliptic sheaves in higher rank, expressed in terms of class numbers.
For the case where the base field is a rational function field with a finite constant field, an analogue of the Hurwitz–Kronecker class number relation (for Drinfeld modular curves) has been derived in \cite{JKY} and \cite{WgY} through a detailed study of Drinfeld modular polynomials in \cite{Bae}, \cite{B-L}, and \cite{Hsia}.
Additionally, a Hirzebruch–Zagier-type formula was established in \cite{GWei0}, which relates the intersection numbers of Hirzebruch–Zagier-type divisors on Drinfeld–Stuhler modular surfaces to modified Hurwitz class numbers through a bridge established from a particular theta integral, which is a ``Drinfeld-type'' automorphic form.
This paper represents our initial effort to extend these ideas to the ``higher rank'' setting. To streamline the discussion, we focus on the case of modular varieties of $\cD$-elliptic sheaves where $\cD$ is a ``maximal order'' of a central division algebra $D$ of dimension $r^2$ over a global function field $k$, with $r$ being a prime number distinct from the characteristic of $k$.

Fix a place $\infty$ of the global function field $k$, which we regard as the place at {\it infinity}. Let $k_{\infty}$ denote the completion of $k$ at $\infty$, and let $\bC_{\infty}$ be the completion of an algebraic closure of $k_{\infty}$.
Let $X$ be the (coarse) moduli scheme of $\cD$-elliptic sheaves defined over $\bC_{\infty}$. For each nonzero ideal $\fa$ of $A$, where $A$ is the ring of elements in $k$ that are regular away from $\infty$, we define an $(r-1)$-cycle $\cZ_{\fa}$ on $X \times X$ in \eqref{eqn: Za} arising from the Hecke correspondences on $X$ (analogous to the classical case). In particular, the cycle $\cZ := \cZ_1$ represents the diagonal cycle of $X$ on $X \times X$.
The main theorem of this paper is stated as follows:

\begin{theorem}\label{thm: MT}
{\rm (Theorem~\ref{thm: C-R})} Keep the above notation and the assumption that $r$ is a prime number distinct from the characteristic of $k$.
Then for each nonzero ideal $\fa$ of $A$, the intersection number of $\cZ$ and $\cZ_{\fa}$ is equal to
\[
i(\cZ\cdot \cZ_{\fa}) =
\frac{r}{q-1}\cdot \left(\sum_{\vec{c}} H^D(\vec{c})
+\sum_{c}H^D(0)\right),
\]
where:
\begin{itemize}
    \item $q$ is the cardinality of the constant field of $k$;
    \item the first sum runs through all $\vec{c}=(c_1,...,c_r) \in A^r$ with $c_r \cdot A = \fa$ so that $f_{\vec{c}}(x):= x^r+c_1x^{r-1}+\cdots+ c_r \in A[x]$ is irreducible over $k$ and the field $K_{\vec{c}}:=k[x]/(f_{\vec{c}}(x))$ is imaginary over $k$ (i.e.\ $\infty$ does not split in $K_{\vec{c}}$);
    \item the second sum runs through all $c \in A$ with $c^r \cdot A = \fa$;
    \item $H^D(\vec{c})$ is the modified Hurwitz class number of $R_{\vec{c}}:= A[x]/f_{\vec{c}}(x)$ with respect to $D$ in Definition~\ref{def: MHC};
    \item $H^D(0)$ is a ``volume quantity'' introduced in \eqref{eqn: H(c)}.
\end{itemize}
\end{theorem}

\begin{remark}
As $D$ is a division algebra, there is no ``cuspidal intersection'' in our situation.
Also, the assumption that $r$ is a prime ensures that the components in the intersection $\cZ \cdot \cZ_{\fa}$ have dimensions either $0$ or $r-1$.
For composite $r$, the intersection behavior may involve components associated with moduli schemes of $\cD'$-elliptic sheaves, where $\cD'$ is a non-maximal order in $D'$, and $D'$ is a central simple algebra over a finite extension $k'$ of $k$, which may not be a division algebra.
Furthermore, the condition that $r$ is distinct from the characteristic of $k$ ensures that the imaginary fields corresponding to the intersection points are separable over $k$. This separability guarantees the transversality of the intersection behavior of the pullback of these cycles in suitable finite coverings, enabling the intersection number to be determined simply by counting the intersection points of the pullback cycles.
Without these assumptions, the overall structure of the intersections becomes significantly more complex and technical, which is why we impose these conditions in our first attempt to study higher rank cases in this paper.
\end{remark}

\subsection{Outline of the proof of Theorem~\ref{thm: MT}}

As the (coarse) modular variety $X$, which parametrizes \( \cD \)-elliptic sheaves, may not be smooth over \( \bC_\infty \), we use Briney's theory to study the intersections in question.
To address this, we consider a finite Galois covering \( X(\fn) \) of  $X$, where \( X(\fn) \) is the modular variety of \( \cD \)-elliptic sheaves equipped with a level \( \fn \)-structure for a suitable nonzero ideal \( \fn \) of  $A$. This covering ensures that \( X(\fn) \) is smooth over \( \bC_\infty \).
Using the projection formula, we can reformulate the intersection number of interest in terms of the corresponding pullback cycles on \( X(\fn) \times X(\fn) \) as follows (see Theorem~\ref{ex: proj-formula}):
\begin{equation}\label{eqn: PJF}
i(\cZ \cdot \cZ_{\fa}) = \frac{1}{(q-1)
\cdot [X(\fn):X]}\sum_{\tilde{g}}i(\cZ(\fn)\cdot \cZ(\fn,g)),
\end{equation}
where $[X(\fn):X]$ is the degree of the ``Galois covering'' $X(\fn)$ over $X$; $\cZ(\fn)$ is the diagonal cycle of $X(\fn)$ on $X(\fn)\times X(\fn)$; $\cZ(\fn,g)$ is introduced in \eqref{eqn: Zng}; and the sum runs through a double coset space corresponding to the Hecke correspondence associated with $\fa$ on $X(\fn)$.

Next, the rigid-analytic structure of these varieties allows us to establish the transversality of the intersection \( \cZ(\fn) \cdot \cZ(\fn, g) \) when \( \cZ(\fn) \) and \( \cZ(\fn, g) \) are distinct. This transversality ensures that \( i(\cZ(\fn) \cdot \cZ(\fn, g)) \) can be computed simply by counting the number of their intersection points.

Furthermore, through a sequence of double-coset comparisons, we relate these intersection numbers to the number of optimal embeddings of the corresponding imaginary orders into \( \cD^\infty \) (see Corollary~\ref{cor: emb}). Additionally, the self-intersection number of \( \cZ(\fn) \) is shown to equal the Euler–Poincaré characteristic of \( X(\fn) \), which can be interpreted as a volume quantity via an analogue of the Gauss–Bonnet formula established by Kurihara in \cite{Ku} (see Proposition~\ref{prop: self-int}).

	Finally, by applying Eichler’s theory, we express these intersection numbers in terms of modified Hurwitz class numbers, and complete the proof of Theorem~\ref{thm: MT}.

\begin{remark}
The recent groundbreaking works of Feng, Yun and Zhang in \cite{FYZ1} and \cite{FYZ2} establish a ``higher'' Siegel--Weil formula in the function field setting, showing a function field analogue of the Kudla–Rapoport conjecture.
They express the non-singular Fourier coefficients of all central derivatives of Siegel–Eisenstein series on unitary groups in terms of the degrees of corresponding special cycles on the moduli stack of unitary shtukas.
In contrast, our result takes a complementary perspective in the case where the group is $D^*$ (an inner form of $\GL_r$): we write the intersection numbers of specific cycles on the modular varieties of $\cD$-elliptic sheaves as combinations of the class numbers of imaginary orders, which are natural arithmetic invariants.

Notably, when $r = 2$, we would be able to adapt the method in \cite{GWei0} and \cite{GWei}, utilizing the Weil representation, to explicitly construct an automorphic form whose Fourier coefficients correspond to the intersection numbers in question.
Although this method does not appear to generalize to $r > 2$, we believe that a deeper connection with automorphic forms on $D^*$ persists and remains to be fully understood. This intriguing connection will be investigated in future work.
\end{remark}

\subsection{Content}

This paper is organized as follows. Section~\ref{sec: No} introduces the basic notation and preliminaries. In Section~\ref{sec: D-ell}, we recall the definition of $\cD$-elliptic sheaves and discuss level structures. The moduli spaces of $\cD$-elliptic sheaves and the Hecke correspondences are presented in Section~\ref{sec: M-Dell}. Section~\ref{sec: Briney} reviews Briney’s extension of intersection theory, leading to the derivation of equality~\eqref{eqn: PJF} in Section~\ref{sec: Proj-F}.

The transversality of the non-self-intersection $\cZ(\fn) \cdot \cZ(\fn,g)$ is verified in Section~\ref{sec: tran-int} from a rigid-analytic perspective. In Section~\ref{sec: count}, we perform several technical double-coset comparisons to connect the intersection points to optimal embeddings of imaginary orders and to express the self-intersection of $\cZ(\fn)$ in terms of a volume quantity. Finally, Section~\ref{sec: CR} connects the intersections in question with the modified Hurwitz class numbers and completes the proof of Theorem~\ref{thm: MT}. For completeness, a brief review of Eichler's theory of optimal embeddings is included in Appendix~\ref{sec: Eich}.

\subsubsection*{Acknowledgments}
Part of this work was carried out during the first author’s visit to the National Center for Theoretical Sciences (NCTS).
The authors sincerely thank the NCTS for its generous support of this collaboration.
The first author is supported by Academia Sinica Investigator Grant AS-IA-112-M01 and NSTC grant 113-2115-M-001-001.
The second author is supported by the NSTC grant 109-2115-M-007-017-MY5, 113-2628-M-007-003, and the NCTS.

\section{Notation}\label{sec: No}

Let $\bF_q$ denote the finite field with $q$ elements.
Let $C$ be a geometrically connected smooth projective curve over a finite field $\bF_q$ with function field $k$. We denote the set of closed points of $C$ by $|C|$, identifying with the set of places of $k$. Fix one closed point $\infty$ of $C$, referred to the place at infinity. Let $A:=\Gamma(C\setminus\{\infty\}, \cO_C)$ denote the ring of regular functions outside of $\infty$ where $\cO_C$ is the structure sheaf of $C$.

Given $v\in |C|$, the completion of $k$ at $v$ is denoted by $k_v$, and let $\cO_v$ be the maximal compact subring of $k_v$. 
We define the adele ring of $k$ to be the restricted product of the completions $k_v$ at the places of $k$: $$\bA:=\prod_{v\in |C|} \!\! {}^{\mathlarger{'}}\,\, k_v$$
\noindent and the ring of finite adeles as the restricted product of the completions $k_v$ at the finite places:
$$\bA_f:=\prod_{v\in |C|\setminus \{\infty \}} 
\hspace{-0.5cm} {}^{\mathlarger{'}}\,\, k_v.$$
\noindent The maximal compact subring of $\bA_f$ is $\cO^{\infty}:= \prod_{v\in |C|\setminus \{\infty \}} \cO_v$.

Let $D$ be a central division algebra over $k$ of dimension $r^2$.
We say that $D$ is {\it ramified at a place $v$} if $D_v:= D\otimes_k k_v$ is not isomorphic to the matrix algebra $\bM_r(k_v)$.
Let $\bf{Ram}$ denote the set of ramified places of $D$.
Here we assume that $\infty \notin \bf{Ram}$.
Let $\cD$ be a locally free sheaf of $\cO_C$-algebras such that the stalk of $\cD$ at the generic point $\eta$ of $C$ is the division algebra $D$, i.e, $\cD_{(\eta)}=D$.
Throughout this paper, we always assume that $\cD$ is a maximal $\cO_C$-order in $D$.
For a place $v$ of $k$, we define
\noindent  $$\cD_v:=\cD_{(v)} \otimes_{\cO_{C,v}} \cO_v,$$
where $\cD_{(v)}$ (resp.~$\cO_{C,v}$) is the stalk of $\cD$ (resp.~$\cO_C$) at $v$. We put 
\[
\cD^{\infty}=\prod_{v\in |C|\setminus \{\infty \}} \cD_v.
\]
In particular, our assumption implies that $\cD_\infty \cong \bM_r(\cO_\infty)$, the matrix ring with entries in $\cO_{\infty}$.
Finally, let $\cO_D:=\Gamma(C \setminus \{\infty \}, \cD)$ be the global sections of $\cD$ away from $\infty$.
Then
\[
\cD^{\infty}= \cO_D\otimes_A \cO^{\infty}.
\]

\section{\texorpdfstring{$\cD$-elliptic sheaves}{D-elliptic sheaves}}
\label{sec: D-ell}

In this section, we recall the definition and some basic facts about $\cD$-elliptic sheaves. For more details on $\cD$-elliptic sheaves, please refer to \cite{LRS}. 
\subsection{Defintion of \texorpdfstring{$\cD$}{D}-elliptic sheaves}
First, we introduce some notations. For an $\bF_q$-scheme $S$, we denote the \emph{Frobenius endomorphism} on $S$ by
$$\sigma_S : S \longrightarrow S,$$ 
which is defined as the identity on points and as the $q$-power map on the structure sheaf. Note that for an $\bF_q$-scheme $S$, one can form the fiber product $C \times_{\bF_q} S$ over $\bF_q$. We will assume this is understood over $\bF_q$ and denote the fiber product simply by $C \times S$. Let 
$$\sigma := \operatorname{id}_C \times \sigma_S : C \times S \longrightarrow C \times S,$$ 
i.e., $\sigma$ acts on $C$ as the identity and on $S$ as the Frobenius endomorphism $\sigma_S$.

\begin{definition}\label{def-D-ellsh}
Let $S$ be a scheme over $\bF_q$. A $\cD$-elliptic sheaf over $S$ is a pair $(\ul{\cE}, \psi)$ where $\psi: S \longrightarrow (C \setminus {\bf Ram})$ is a morphism  and 
$\ul{\cE}$ is a  sequence $(\cE_i, j_i, t_i), i\in \bZ$. Here each $\cE_i$ is a locally free $\cO_{C \times S}$-modules of rank $r^2$ with a right $\cD$-action which is $\cO_{C\times S}$-linear and the restriction of $\cD$ to the scalars is same as the $\cO_C$-action. And the morphisms

$$j_i: \cE_i \longrightarrow \cE_{i+1}$$
$$t_i: \sigma^*\cE_i \longrightarrow \cE_{i+1}$$
\noindent are injective $\cO_{C \times S}$-linear morphisms which are compatible with the $\cD$-action such that the following diagram commutes 

\centerline{
\xymatrix{ & \cE_{i}  \ar@{^{(}->}[r]^{j_i} &  \cE_{i+1}\\
\sigma^* \cE_{i-1}  \ar@{^{(}->}[r]_{\sigma^* j_{i-1}} \ar@{->}[ur]^{t_{i-1}} &  \sigma^* \cE_{i} \ar@{->}[ur]^{t_i}
}
}

\noindent Moreover, for each $i\in \bZ$ the following conditions hold:

\begin{enumerate}
    \item (periodicity) Put $\ell = r\cdot \deg \infty$ We have
$$\cE_{i+\ell} = \cE_i(\infty)$$
\noindent where $\cE_i(\infty) = \cE_i \otimes_{\cO_{C \times S}}
 (\cO_C(\infty) \boxtimes \cO_S)$

    \item \label{cokerj} The $\coker j_i$ is supported on $\{\infty \} \times S$ and it is locally free $\cO_S$-module of rank $r$.
    
    \item The $\coker t_i$ has support on the graph of $\psi$ and is locally free of rank $r$ over $S$.

\end{enumerate}

The morphism $\psi$ in the definition of a $\cD$-elliptic sheaf is called the \emph{characteristic} of the $\cD$-elliptic sheaf.
For simplicity, we denote the $\cD$-elliptic sheaf by $\ul{\cE}$ if no confusion arises.
\end{definition}

A \emph{morphism} $f: \ul{\cE} \longrightarrow \ul{\cE}'$  between two $\cD$-elliptic sheaves $\ul{\cE}=(\cE_i, j_i, t_i)$ and $\ul{\cE}'=(\cE'_i, j'_i, t'_i)$ over $S$ is a sequence of morphisms of locally free $\cO_{C\times S}$-modules $f_i: \cE_i \longrightarrow \cE'_i$ which are compatible with $\cD$-action and with the morphisms $j_i$'s and $t_i$'s.

\begin{remark}\label{rem: Dellaffine}
    
Let $\ul{\cE}=(\cE_i, j_i, t_i)$ be a $\cD$-elliptic sheaf over $S$. By the condition (\ref{cokerj}) in Definition \ref{def-D-ellsh}, $H^0((C\setminus \infty)\times S, \cE_i)$ is independent of $i$. Moreover, if $S= \Spec R$ for a ring $R$, it becomes an $R[\tau]$-module where the $\tau$-action comes from the morphism $t_i: \sigma^* \cE_i \longrightarrow \cE_{i+1}$. This module is referred to as a \emph{Drinfeld-Stuhler $\cO_D$-module}. The concept was first implicitly introduced in \cite[Section 3]{LRS}, with further studies available in \cite{Pap}.

\end{remark}

Let $\Ell_{C,\cD}(S)$ denote the category whose objects are the $\cD$-elliptic sheaves over $S$ and whose morphisms are isomorphisms of $\cD$-elliptic sheaves. If $S' \longrightarrow S$ is a morphism of $\bF_q$-schemes, then taking pullback of a $\cD$-elliptic sheaf over $S$ gives us a $\cD$-elliptic sheaf over $S'$. This defines a fibered category 
$$S \longrightarrow \Ell_{C, \cD}(S)$$

\noindent over the category of $\bF_q$-schemes $Sch_{\bF_q}$. Moreover this gives us a stack, denoted by $\Ell_{C, \cD}$, with respect to fppf-topology. The proof follows as a special case of \cite[Proposition 2.12]{Ul}. 

\subsection{\texorpdfstring{$\cD$-elliptic sheaves with level structures}{D-elliptic sheaves with level structures}}

Now we will define level structures on $\cD$-elliptic sheaves. First, we want to recall the following definition:

\begin{definition}
	Let $X$ and $Y$ be two schemes, and let $pr_1: X \times Y \longrightarrow X$, $pr_2: X \times Y \longrightarrow Y$ be the natural projections.
    Given a locally free $\cO_X$-module $\cF$ and a locally free $\cO_Y$-module $\cG$, we define the \emph{external tensor product} of $\cF$ and $\cG$ as follows:	
	$$\cF \boxtimes \cG := pr_1^*(\cF) \otimes_{\cO_{X \times Y}} pr_2^*(\cG),$$	
    which is naturally a locally free $\cO_{X \times Y}$-module.
\end{definition}

Let $I$ be a finite closed subscheme of $C\setminus ({\bf Ram}\cup\{\infty\} \cup \operatorname{Im}\psi)$. Then, the restrictions $\cE_i|_{I\times S}$ are all isomorphic by the morphisms $j_i$'s (cf. Remark \ref{rem: Dellaffine}). Therefore it is independent of the chosen $i$, and we will denote this restriction simply by $\cE|_{I \times S}$.
We write $\cD_I$ for $\cD \otimes_{\cO_C} \cO_I$ where $\cO_I$ is the structure sheaf of $I$ , and let $H^0(I, \cD_I)$ be the ring of global sections of $\cD_I$. 

\begin{definition}\label{levelI}
	Let $S$ be an $\bF_q$-scheme and $\ul{\cE}=(\cE_i, j_i, t_i)$ be a $\cD$-elliptic sheaf over $S$. 
	Let $I$ be a finite closed subscheme of $C\setminus ({\bf Ram}\cup\{\infty\} \cup \operatorname{Im}\psi)$. A \emph{level $I$-structure} on $\ul{\cE}$ is an isomorphism of $\cO_{I\times S}$-modules 	
	$$\iota: \cD_I \boxtimes \cO_S \xrightarrow{\sim} \cE|_{I \times S}$$
	which is compatible with $\cD$-action and the Frobenius 
    endeomorphism $\sigma_S$ on $S$. 
\end{definition}

 \begin{remark}
    Throughout the paper, when we talk about level $I$-structures, $I$ is always a finite closed subscheme of $C\setminus ({\bf Ram}\cup\{\infty\} \cup \operatorname{Im}\psi)$.
\end{remark}

Let $\Ell_{C, \cD, I}(S)$ denote the category whose objects are $\cD$-elliptic sheaves over $S$ with level $I$-structure and whose morphisms are morphisms of $\cD$-elliptic sheaves that respect the level $I$-structure. As before, by using pullbacks one can define a fibered category $\Ell_{C, \cD, I}$ over $Sch_{\bF_q}$, which is a stack with respect to fppf-topology. For a detailed proof we refer to \cite[Proposition 2.19]{Ul}.

Now, let $\ul{\cE}=(\cE_i, j_i, t_i)$ be a $\cD$-elliptic sheaf over $S$ with level $I$-structure and let $Sch_S$ denote the category of schemes over $S$. One can define $t$-invariant elements functor 
$$E_I: Sch_S \longrightarrow H^0(I, \cD_I)\mbox{-modules}$$
by $T \mapsto (H^0(I\times T, \cE|_{I\times S}))^{t=id}$. We then have the following:

\begin{theorem}(\cite[Lemma 2.6]{LRS}, also see \cite[Theorem 2.22]{Ul})\label{thm: torsor1}
The functor $E_I$ is represented by a finite \'etale scheme over $S$ which is free of rank one over $H^0(I, \cD_I)$.
\end{theorem}

\begin{remark} \label{lem: torsor}(\cite[(4.8), page 240]{LRS})
For connected $S$, the set of $I$-level structures is a torsor over the unit group $D_I^*$. More precisely, Let $S$ be an $\bF_q$-scheme, $I \subset I' \subset C$ be two finite closed subschemes with $\deg(I')>0$. Then, the morphism $$r_{I',I}(S) : \Ell_{X, \cD, I'}(S) \longrightarrow \Ell_{X, \cD, I}(S)$$ 
	\noindent which associates a level $I'$-structure to its restriction gives us a $G_{I', I}$-torsor over $C \setminus I'$, i.e, the finite group $G_{I', I}:=\operatorname{Ker}(H^0(I', \cD_{I'})* \longrightarrow H^0(I, \cD_I)^*)$ acts on the set of level $I'$-structures transitively and freely. 
\end{remark}

\section{\texorpdfstring{Moduli schemes of $\cD$-elliptic sheaves and Hecke correspondences}{Moduli varieties of D-elliptic sheaves}}\label{sec: M-Dell}

We shall introduce the moduli schemes of $\cD$-elliptic sheaves and the cycles associated with the Hecke correspondences.

\subsection{\texorpdfstring{Moduli schemes of $\cD$-elliptic sheaves}{Moduli schemes of D-elliptic sheaves}}

In the previous section, we defined the stack $\Ell_{C, \cD, I}$ of $\cD$-elliptic sheaves with level $I$-structures. In fact, $\Ell_{C, \cD, I}$ is an algebraic stack in the sense of Deligne-Mumford (cf. \cite{DM}).

\begin{theorem}\label{thm-Dell-sch} (\cite[Theorem 4.1 and Theorem 5.1]{LRS})

The stack $\Ell_{C, \cD, I}$ is an algebraic stack in the sense of Deligne-Mumford which is smooth of relative dimension $(r - 1)$ over $C \setminus (I\cup \{\infty \})$. Moreover, if $I \neq \emptyset $, it is actually a quasi-projective scheme. 
\end{theorem}

\begin{remark}
    Let $\emptyset \neq I$ be a finite closed subscheme of $C$ such that $I \cap \{ \infty \} = \emptyset$. There exists a nonzero ideal $\fn_I$ of $A$ associated with this closed subscheme. Conversely, if we are given a nonzero ideal $\fn$ of $A$, we get a corresponding finite closed subscheme $I_{\fn}$ of $C$ away from $\infty$.

\end{remark}

\begin{definition}
Let 
$\fn$ be a nonzero ideal of $A$ with $v \nmid \fn$ for every $v \in {\bf Ram}$ and $I_\fn$ be the  finite closed subscheme of $C$ associated with $\fn$. We denote the representing scheme of $\Ell_{C, \cD, I_\fn}/\bZ$ by $X(\fn)$.

\end{definition}

Moreover, different from Drinfeld modular varieties, we have the following theorem: 

\begin{theorem} 
(\cite[Theorem 4.4.9]{BS}, see also \cite[Theorem 6.1]{LRS})

    The morphism 
    $$X(\fn) \longrightarrow C \setminus (I_\fn \cup {\bf Ram})$$
\noindent is proper.
\end{theorem}

\begin{remark}
We want to note that in \cite{LRS}, it is assumed that the characteristic of $\cD$-elliptic sheaf is away from $\infty$. However, both theorems above hold in general. We refer to \cite[Section 5 and Section 6]{Ul} for the proof in general case.
Moreover, the base change of $X(\fn)$ to $\bC_{\infty}$ is actually a projective scheme (see \cite[p.~218]{LRS}).
\end{remark}

We denote by $\wh{X}(\fn)$ the formal completion of the scheme $X(\fn)$ along the fiber over $\infty$, which is a formal scheme over $Spf(\cO_{\infty})$.

Let $\Omega_r$ denote the Drinfeld symmetric space:
$$\Omega_r:= \bP^{r-1}(\bC_{\infty})\setminus \cup (k_{\infty}\mbox{-rational hyperplanes}).$$
It has a rigid analytic stucture, and thus becomes a rigid space in the sense of Raynaud. By taking a formal completion over the fiber over $\infty$, we obtain a formal scheme $\widehat\Omega_r$ (\cite[Remark 4.3.1]{BS}). For more on the Drinfeld symmetric space, we refer to \cite{Dr1}, \cite{DH} and \cite{SS}.

Recall that
$\cD^{\infty} = \cO_D\otimes_A \cO^{\infty} = \prod_{v \neq \infty} \cO_D \otimes_A \cO_v$.
Set
$D^*(\bA_f):=\big(D\otimes_k \bA_f\big)^*$
and
\[
\cK(\fn)
:= 
\{a \in (\cD^{\infty})^* \mid a \equiv 1 \bmod \fn \}.
\]
Then the following theorem holds:

\begin{theorem} \label{thm-unif}
One has an isomorphism of formal schemes 
$$ \wh{X}(\fn) \simeq D^* \backslash \wh{\Omega}_r \times D^*(\bA_f)/ \cK(\fn).$$
\end{theorem}

\begin{proof}
\cite[Remark 4.4.11]{BS}, \cite[Remark 16.11]{Ul}
\end{proof}

Let $\cK^\infty$ be an open compact subgroup of $D^*(\bA_f)$.
We want to look at the quotient $D^* \backslash \wh{\Omega}_r \times D^*(\bA_f)/ \cK^\infty$ closer. 
Let $d_i \cK^\infty$, $i=1,...,s$ denote the representatives of the double coset space $D^* \backslash D^*(\bA_f) /\cK^\infty$ and let
$$\stab_{D^*}(d_i\cK^\infty) = \{d \in D^* \mid dd_i\cK^\infty=d_i\cK^\infty \}$$
\noindent be the stabilizer of $d_i\cK^\infty$ with respect to $D^*$-action.
In other words, 
$$d \in \stab_{D^*}(d_i\cK^\infty) \iff dd_i\cK^\infty=d_i\cK^\infty \iff d \in d_i \cK^\infty d_i^{-1}$$
\noindent so $d \in d_i \cK^\infty d_i^{-1} \cap D^*$. Put $\Gamma_i := d_i \cK^\infty d_i^{-1} \cap D^*$, which is a discrete subgroup of $D_\infty^* \cong \GL_r(k_\infty)$.

\begin{proposition}\label{prop: dc}
Keep notations as above.
We have

    $$D^* \backslash \wh{\Omega}_r \times D^*(\bA_f)/ \cK^\infty \simeq \coprod_{i=1}^{s} \Gamma_i \backslash \widehat \Omega_r.$$
\end{proposition}

\begin{proof}

Write $D^*(\bA_f)$ as the finite disjoint union $\coprod_{i=1}^s D^*d_i\cK^\infty$. Then

$$\widehat \Omega_r \times D^*(\bA_f) = \coprod_{i=1}^s \widehat \Omega_r \times D^*d_i\cK^\infty.$$
Hence we can identify $D^* \backslash \wh{\Omega}_r \times D^*(\bA_f)/ \cK^\infty$ with $\coprod_{i=1}^d \Gamma_i \backslash \widehat \Omega_r$.

\end{proof}

By construction, the generic fibre over $k_{\infty}$ of $\widehat \Omega_r$  is the rigid space $\Omega_r$. Therefore, together with the Proposition \ref{prop: dc}, we have the following proposition:

\begin{proposition} \label{prop: rigid}
Let $\fn$ be a nonzero ideal of $A$ and take $\cK^\infty = \cK(\fn)$.
We have    $$X(\fn)(\bC_\infty)\simeq \widehat X(\fn)^{an} \simeq \coprod_i \Gamma_i \backslash \Omega_r \simeq D^*\backslash \Omega_r\times D^*(\bA_f)/\cK(\fn).$$
\end{proposition}

\subsection{Hecke correspondences and associated cycles}\label{sec: Hecke}
\label{sec-cycles}

In this section we will give a brief summary of Hecke correspondences on the modular variety of $\cD$-elliptic sheaves. For details we refer to \cite[Section 7]{LRS}. 

Let $I$ be a finite closed subscheme of $C$ such that $I \cap (\{\infty \} \cup {\bf Ram}) = \emptyset$. Define

$$\Ell_{C, \cD}^{\infty}=\varprojlim_{I} \Ell_{C, \cD, I}.$$

For an $\bF_q$-scheme $S$, a section of $\Ell_{C, \cD}^{\infty}$ over $S$ consists of triples $(\cE_i, j_i, t_i)$ as in the Definition \ref{def-D-ellsh} with the following `level structure', i.e, a $\cD$-linear isomorphism
$$(\cD \otimes_{\cO_X} \cO^{\infty}) \boxtimes \cO_S \xrightarrow{\sim} \cE_i \otimes_{\cO_{C \times S}} (\cO^{\infty} \boxtimes \cO_S)$$
\noindent where $\cO^{\infty}=\prod_{v \neq \infty} \cO_v$. 

There is a right action of $D^*(\bA_f) \cap \cD^{\infty}$ on $\Ell_{C, \cD}^{\infty}$, extending the action of $(\cD^{\infty})^*$ (cf. \cite[(7.4)]{LRS}). 
Let $\cK^{\infty} \subseteq D^*(\bA_f)$ be an open compact subgroup and fix $g \in D^*(\bA_f)$. Then, we get a correspondence over $\Spec k$:
\begin{equation}\label{eqn: HC}	\xymatrix{ 
		& \Ell_{C, \cD}^{\infty} / \cK^{\infty} \cap g^{-1}\cK^{\infty}g \ar@{->}[dr]^{\pi_{g, 2}} \ar@{->}[dl]_{\pi_{g, 1}} & \\
		\Ell_{C, \cD}^{\infty} / \cK^{\infty}  & & \Ell_{C, \cD}^{\infty} / \cK^{\infty}
}
\end{equation}
\noindent where the morphisms $\pi_{g, 1}$ and $\pi_{g, 2}$ are induced, respectively, by the inclusions 
\[
\cK^{\infty} \cap g^{-1} \cK^{\infty} g \subseteq \cK^{\infty}
\quad \text{ and }\quad 
\cK^{\infty} \cap g^{-1}\cK^{\infty} g \xrightarrow{Ad(g)} \cK^{\infty}.
\]

Given a nonzero ideal $\fn$ of $A$ with $v \nmid \fn$ for every $v \in {\bf Ram}$, recall that we let
$$\cK(\fn)=\{a \in (\cD^{\infty})^* \mid a \equiv 1 \mbox{ mod } \fn \}.$$

\noindent For every $g \in D^*(\bA_f)$, define $$\cK(\fn, g)=\cK(\fn)\cap g^{-1} \cK(\fn)g.$$
We also put 
\[
\cK:=\cK(1) (=(\cD^{\infty})^*)
\quad \text{ and } \quad 
\cK_g := \cK(1,g) = \cK \cap g^{-1}\cK g.
\]

Let $X(\fn,g)$ be the scheme corresponding to $\bZ \backslash \Ell_{C,\cD}^\infty/\cK(\fn,g)$.
Then the correspondence (\ref{eqn: HC}) gives us the following:
\[
\centerline{
	\xymatrix{ 
		& X(\fn,g) \ar@{->}[dr]^{\pi_{\fn,g, 2}} \ar@{->}[dl]_{\pi_{\fn, g, 1}} & \\
		X(\fn)  & & X(\fn)
}}
\]
In particular, similar to Proposition~\ref{prop: rigid}, we may identify
\[
X(\fn,g)(\bC_\infty) \cong D^*\backslash \Omega_r \times D^*(\bA_f)/\cK(\fn,g).
\]
Then the two morphisms $\pi_{\fn,g,1},\pi_{\fn,g,2}:X(\fn,g)\rightarrow X(\fn)$
can be realized as follows:
for every $z \in \Omega_r$ and $b \in D^*(\bA_f)$, let $[z,b]_{\fn,g}$
and $[z,b]_{\fn}$be the representing double coset in $D^*\backslash \Omega_r \times D^*(\bA_f)/\cK(\fn,g)$
and in 
$D^*\backslash \Omega_r \times D^*(\bA_f)/\cK(\fn)$,
respectively.
Then 
\[
\pi_{\fn,g,1}([z,b]_{\fn,g}) = [z,b]_{\fn}
\quad \text{ and } \quad 
\pi_{\fn,g,2}([z,b]_{\fn,g}) = [z,bg^{-1}]_{\fn}.
\]
Therefore we have a morphism
$$\pi_{\fn, g}: X(\fn, g) \longrightarrow X(\fn) \times X(\fn)$$
\noindent which is realized by 
\begin{equation}\label{eqn: pi-n-g}
[z, b]_{\fn, g} \mapsto \Big(\pi_{\fn, g, 1}([z, b]_{\fn}), \pi_{\fn, g, 2}([z, b]_{\fn})\Big) = \big([z,b]_{\fn},[z,bg^{-1}]_{\fn}\big)
\end{equation}
\noindent for every $z \in \Omega_r, b \in D^*(\bA_f)$.

By abuse of notations, we still denote by $X(\fn,g)$ its base change to $\bC_{\infty}$.
Recall that it is known that $X(\fn,g)$ has constant dimension $r-1$ over $\bC_{\infty}$ (cf. Theorem \ref{thm-Dell-sch}).
Let $\cX(\fn,g)$ be the $(r-1)$-cycle of $X(\fn,g)$ associated with itself.
We define 
\begin{equation}\label{eqn: Zng}
\cZ(\fn, g) = \pi_{\fn, g, *}(\cX(\fn, g))
\end{equation}
which is an $(r-1)$-cycle of $X(\fn)\times X(\fn)$.
In particular, we denote
\[
\cZ(\fn) =  \cZ(\fn, 1),
\quad 
\cZ_g = \cZ(1, g),
\quad 
\text{and}
\quad 
\cZ=\cZ_1.
\]

Moreover, let $\fa$ be a nonzero ideal of $A$. 
We define the following associated cycle on $X(1)\times X(1)$:
\begin{equation}\label{eqn: Za}
\cZ_{\fa} = \sum_{[g]} \cZ_g
\end{equation}
\noindent where the sum is taken over double cosets $[g] \in \cK\backslash (\cD^{\infty} \cap D^*(\bA_f))/ \cK$ with the condition $\hbox{Nr}(g) \cO^{\infty} \cap k = \fa$. 
Here $\hbox{Nr}$ denotes the natural extension of the reduced norm of $D$ to $D^*(\bA_f)$.  

The main goal of this paper is to study the intersection cycle $\cZ\cdot\cZ_{\fa}$,  
and to connect the intersection number $i(\cZ\cdot\cZ_{\fa})$ with the class numbers of ``imaginary orders''.

\section{Briney's Theory} \label{sec: Briney}

In this section we briefly recall Briney's theory of intersections of quotients of algebraic varieties in \cite{B}. Here we assume all varieties are (absolutely) irreducible and abstract as in \cite{B}. 

\subsection{Quotient varieties by finite groups}

Fix an algebraically closed field $L$. We will focus on quotients of an algebraic variety $X$.  Let $\fg$ be a finite subgroup of the group of all automorphisms of $X$. Then one can look at the quotient $X/\fg$ of $X$ by $\fg$. We assume that $X/\fg$ is an algebraic variety, i.e, every $\fg$-orbit is contained in an affine open set in $X$ (\cite[Section 1.4]{Serre}). Then the natural projection map $\lambda: X \longrightarrow X/\fg$ is a proper morphism. Moreover it is a covering map. 
If all $\sigma \in \fg$ are defined over $L$ and $L$ is a field of definition of $X$, we say that $(X, \fg)$ is defined over $L$.

Let $(X, \fg)$ and $(X', \fg')$ be as in the previous paragraph. We say they are equivalent if there are surjective isomorphisms $\alpha: X \longrightarrow X'$ and $\beta: \fg \longrightarrow \fg'$ such that $\alpha(x^{\sigma})=\alpha(x)^{\beta(\sigma)}$ for all $x \in X$ and for all $\sigma \in \fg$. In this case, $\alpha$ gives us isomorphism of quotients $\alpha': X/\fg \longrightarrow X'/\fg'$ such that $\lambda'\circ \alpha= \alpha' \circ \lambda$.

Now, let $X, X'$ and $Y$ be varieties such that there are morphisms $\theta: X \longrightarrow Y$ and $\theta': X' \longrightarrow Y$. Let $\fg$ and $\fg'$ be finite subgroups of the automorphism groups of $X$ and $X'$, respectively. Then, each morphism splits through the associated quotient:

\centerline{
	\xymatrix{ 
		X \ar@{->}[dr]^{\lambda} \ar@{->}[rr]^{\theta} & & Y  & & X' \ar@{->}[dr]^{\lambda'}  \ar@{->}[rr]^{\theta'}& & Y \\
		  & X/\fg \ar@{->}[ur]_{\tau} & & & & X'/\fg' \ar@{->}[ur]_{\tau'}
}}

We say $(X, \fg, \theta)$ and $(X', \fg', \theta')$ are equivalent if $(X, \fg)$ and $(X', \fg')$ are equivalent and if $\tau = \tau' \circ \alpha' $. In other words every triangle in the following diagram commutes:

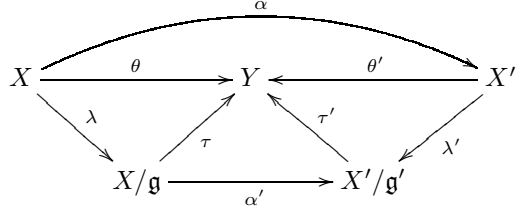
\centerline{
	\xymatrix{ 
		X \ar@{->}[dr]^{\lambda} \ar@{->}[rr]^{\theta} \ar@/^2pc/[rrrr]^{\alpha} & & Y  & & X' \ar@{->}[dl]^{\lambda'}  \ar@{->}[ll]_{\theta'}& &  \\
		  & X/\fg \ar@{->}[ur]_{\tau} \ar@{->}[rr]_{\alpha'} &  & X'/\fg' \ar@{->}[ul]_{\tau'}
}}

Now, we can define a \emph{quotient structure}. 

\begin{definition}\label{defn-quotstr}
A quotient structure on the variety $Y$ is the equivalence classes of the `real' quotients. More precisely, a quotient structure is the equivalence classes of the triples $(X, \fg, \theta)$ such that $\tau: X/ \fg \longrightarrow Y $  is an isomorphism (\cite[Definition 2.1]{B}). 
In this case we say $Y$ is a \emph{quotient variety of $X$ by $\fg$} and is denoted by $Y \longleftarrow [X, \fg, \theta]$. 
    
\end{definition}

Let $X$ be an algebraic variety and $Y \longleftarrow [X, \fg, \theta]$ be a quotient variety of $X$ by $\fg$ as above.
For a subvariety $A \subseteq Y$, a typical component of $\theta^{-1}(A)$ will be denoted by $A'$. In this case we say $A'$ \emph{lies over} $A$. For each $A' \subseteq X$, we define the following two groups:

\begin{enumerate}
    \item \emph{The splitting group} of $A'$ is defined as $$\fg^s(A')=\{ \sigma \in \fg  \mid (A')^{\sigma}=A' \}.$$

    \item \emph{The inertia group} of $A'$ is the group 
    $$\fg^i(A')=\{ \sigma \in \fg \mid x^{\sigma}=x \hbox{ for all } x \in A' \}.$$
    
\end{enumerate}

For any subvariety $A$ of the quotient $Y$ we associate certain numerical characters where $A'$ denotes a typical component of $\theta^{-1}(A)$:

\begin{enumerate}
    \item \emph{The degree of $A'$ over $A$} is defined as $$d(A)=[A' : A].$$

    \item The degree
    $$d_s(A)=[A': A]_s$$
    \noindent is called \emph{separable degree of $A'$ over $A$}. 

    \item We define \emph{the inseparable degree of $A'$ over $A$} as $$d_i(A)=[A' : A]_i.$$

    \item Define $$\ell(A)= [\fg^i(A') : 1] / d_i(A),$$
     which is analogous to the ramification index for a valuation ring. 
\end{enumerate}

\begin{remark}\label{rem: numeq}
Let $n(A)$ denote the number of components of $\theta^{-1}(A)$ and $g:=[\fg:1]$. We have the relation
$$n(A)d(A)\ell(A)=g=n(A)d_s(A)[\fg^i(A'):1]$$
\noindent for any $A \subseteq Y$ and any $A' \subseteq X$ lying over $A$.

Note that the components of $\theta^{-1}(A)$ are permuted transitively by the elements of $\fg$ (\cite[Ch. V, Proposition 3 on p. 189]{Chev}), so that the above characters depend only on $A$ and not on the choice of $A'$.
\end{remark}

\subsection{Intersection multiplicity and Projection formula}

We recall the following theorem:

\begin{theorem}(\cite[Appendix A, Theorem 1.1]{H})
    There is a unique intersection theory for cycles modulo rational equivalence on the varieties which satisfies conditions A1-A7 in \cite[Appendix A]{H}.
\end{theorem}

By this theorem we can talk about the intersection of smooth (quasi-) projective varieties. Briney's theory enables us to extend the intersection theory of finite quotients of smooth (quasi-) projective varieties to not necessarily smooth ones by allowing rational multiplicities. First we define covering cycles. 

Let $X, Y$ be (quasi-)projective (not necessarily smooth) varieties over $L$. And, let $\theta: X \longrightarrow Y$ be the projection map as defined in the Definition \ref{defn-quotstr}. For any subvariety $A \subset Y$, let $A'_1, \cdots, A'_n$ be distinct irreducible components of $\theta^{-1}(A)$. Set

$$A^* = \sum_{j=1}^n A'_j.$$

We call $A^*$ \emph{the covering cycle of $A$}. 

By the projection map $\theta: X \longrightarrow Y$, we get the following mappings on subvarieties

$$\theta_z(A'):= d(A)A$$
\noindent and
$$\theta^z(A):= \ell(A) A^*.$$
One can extend the mappings $\theta_z$ and $\theta^z$ linearly to the cycles. 

\begin{remark}
    For every cycle $Z$ on $Y$ we have    $\theta_z(\theta^z(Z))=g Z$, where $g= [\fg :1]$. This relation is a special case of \emph{projection formula} (cf. Proposition \ref{prop: projform}).
\end{remark}

Let $A, B \subseteq Y$ and $A^*, B^*$ be the corresponding covering cycles. Let $P$ be a proper component of $A \cap B$ on $Y$. 

We can compute the intersection multiplicity $(P; A\cdot B)_Y$ by carrying over the multiplicities given on $X$ via the quotient map $\theta: X \longrightarrow Y$.

\begin{definition}(\cite[Definition 3.2]{B}) \label{def : intmult}

 The \emph{intersection multiplicity} of $A$ and $B$ at $P$ on $Y$ is

 $$(P; A\cdot B)_Y = \frac{\ell(A) \ell(B)}{\ell(P)} (P'; A^*\cdot B^*)_X$$

 \noindent where $P'$ is any component of $P^*$ and the multiplicity $(P'; A^* \cdot B^*)_X$ is on $X$. 
    
\end{definition}

\begin{lemma}(\cite[Lemma 2.6]{B})\label{lem: numbofcompof-cycle}

If $A'\cdot B^*$ is defined on $X$, so is $A'^{\sigma}\cdot B^*$ for all $\sigma \in \fg$, and $\theta_z'(A'^{\sigma}\cdot B^*)=\theta_z(A'\cdot B^*)$. As a result, 
$$\theta_z(A^*\cdot B^*)=n(A)\theta_z(A'\cdot B^*).$$
    
\end{lemma}

\begin{proposition}(\cite[Proposition 3.3]{B})\label{prop: projform}

Let $Z, T$ be cycles on $Y$ and $Z'$ be cycle on $X$.  

\begin{enumerate}
    
    \item If $Z\cdot T$ is defined on $Y$ then $\theta^z(Z)\cdot\theta^z(T)$ is defined on $X$. And in this case, we have
    $$\theta^z(Z\cdot T) = \theta^z(Z) \cdot \theta^z(T).$$

    \item(\emph{Projection formula}) $Z' \cdot \theta^z(T)$ is defined on $X$ iff $\theta_z(Z')\cdot T$ is defined on $Y$. And in this case, we have
    $$\theta_z(Z'\cdot \theta^z(T))=\theta_z(Z')\cdot T.$$

\end{enumerate}

\end{proposition}

\begin{remark}\label{rem: self-int}
Suppose that $X$ is projective and smooth over $\bC_{\infty}$.
Given two cycles $Z_1',Z_2'$ on $X$ with middle dimension which have transversal intersection,
the intersection number 
$i(Z_1'\cdot Z_2')$ is the counting number of their intersection points.
By the moving lemma, this intersection number can be extended to every two cycles $Z_1',Z_2'$ on $X$ with middle dimension.
Therefore for every two cycle $T_1,T_2$ on $Y$ with middle dimension,
take a cycle $Z_1'$ on $X$ so that $\theta_z(Z_1') = d \cdot T_1$ for some $d \in \bN$.
We may define the intersection number $i(T_1\cdot T_2)$ via the above projection formula:
\[
i(T_1\cdot T_2)= \frac{1}{d} \cdot  i(\theta_z(Z_1')\cdot T_2):= \frac{1}{d}\cdot i(Z_1'\cdot \theta^z(T_2)).
\]
\end{remark}

\section{Projection Formula for Hecke cycles}
\label{sec: Proj-F}

Let $\fn$ be a nonzero proper ideal of $A$
such that 
$v\nmid \fn$ for every $v \in {\bf Ram}$ and 
${\rm ord}_v(\fa)\leq {\rm ord}_v(\fn)$ for every $v \notin {\bf Ram}\cup \{\infty\}$. 
The main result of this section is the following theorem:

\begin{theorem}\label{ex: proj-formula}
    $$i(\cZ \cdot \cZ_{\fa}) = \frac{1}{ [\cK :  \cK(\fn)]}\sum_{\tilde{g}} i(\cZ(\fn) \cdot \cZ(\fn, g)) $$
where the sum is taken over $\tilde{g} \in 
\cK(\fn) \backslash (\cD^\infty)^*\cap D^*(\bA_f)/\cK(\fn)
$
with $\text{\rm Nr}(g)\cO^{\infty} \cap k =\fa$.
\end{theorem}

The rest of the section is devoted to the proof this theorem.

\subsection{Projection Formula}

Let $X(1), X_g, X(\fn)$ and $X(\fn, g)$ be defined as before. As discussed in Proposition~\ref{prop: rigid}, we may write
$$X(1)=\coprod_i X(1)_i$$
$$X_g=\coprod_i X_{g, i}$$
$$X(\fn)=\coprod_i\coprod_j X(\fn)_{ij}$$
$$X(\fn, g)=\coprod_i\coprod_jX(\fn, g)_{ij}$$

In order to study the intersection $\cZ(\fn)\cdot \cZ(\fn,g)$, we will be working over product of these varieties. More precisely we have 
$$X(1)\times X(1)=\coprod_iX(1)_i \times \coprod_i X(1)_i, $$
\noindent which can be written as follows: we define $\ul i=(i_1, i_2)$ as the tuple where $i_1$ correspond to first component of $X(1) \times X(1)$ and  $i_2$ correspond to the second component, i.e., we can write
$$X(1)\times X(1)=\coprod_i X(1)_i \times \coprod_iX(1)_i=\coprod_{\ul i}(X(1)_{i_1} \times X(1)_{i_2}).$$

Similarly for $\ul i=(i_1, i_2)$ and $\ul j=(j_1, j_2)$ we can write 
$$X(\fn) \times X(\fn)=\coprod_i\coprod_j X(\fn)_{ij} \times \coprod_i\coprod_j X(\fn)_{ij} = \coprod_{\ul i} \coprod_{\ul j} (X(\fn)_{i_1j_1}\times X(\fn)_{i_2j_2})$$

Recall the definition of $\pi_{\fn, g}$:
$$\pi_{\fn, g}: X(\fn, g) \longrightarrow X(\fn) \times X(\fn)$$
\noindent sending an element $[z,b]_{\fn, g} $ to $([z,b]_{\fn}, [z, bg^{-1}]_{\fn})$, i.e, we have a morphism  
$$\coprod_i \coprod_j X(\fn, g)_{ij} \longrightarrow \coprod_{\ul i}\coprod_{\ul j}(X(\fn)_{i_1j_1} \times X(\fn)_{i_2j_2}).$$
In particular, $\pi_{\fn, g}$ gives a morphism 
$$ X(\fn, g)_{i_0j_0} \longrightarrow X(\fn)_{i_0j_0} \times X(\fn)_{i_2j_2},$$
where $i_2 j_2$ is determined by $i_0 j_0$ and $g$.
Moreover, 
let 
$pr_1: X(\fn)\times X(\fn)\longrightarrow X(\fn)$ be the projection to the first component
(which projects $ X(\fn)_{i_0j_0} \times X(\fn)_{i_2j_2}$
to $X(\fn)_{i_0j_0}$).
Then $pr_1 \circ \pi_{\fn, g}: X(\fn, g) \longrightarrow X(\fn)$ coincides with the original covering map.
So, we have the following relations between these objects:

\centerline{
\xymatrix{
& X(\fn,g)=\coprod_i\coprod_jX(\fn, g)_{ij} \ar@{-}[d] \\
X_g=\coprod_i X_{g, i} \ar@{-}[ur] \ar@{-}[dr] & X(\fn)=\coprod_i \coprod_j X(\fn)_{ij} \ar@{-}[d]\\
& X(1)=\coprod_i X(1)_i
}
}

For $\ul i=(i_1, i_2)$ we define $\cZ_{1,\ul i}$ as the intersection of the cycle $\cZ_1$ on $X(1) \times X(1)$ and the $\ul i$-th component of $X(1) \times X(1)$: 
$$\cZ_{1, \ul i}:=\cZ_1 \cap (X(1)_{i_1} \times X(1)_{i_2}).$$
Similarly, define
$$\cZ_{g, \ul i}:= \cZ_g \cap (X(1)_{i_1} \times X(1)_{i_2}).$$
For $\ul i=(i_1, i_2)$ and $\ul j=(j_1, j_2)$ define
$$\cZ(\fn)_{\ul i\ul j}=\cZ(\fn) \cap (X(\fn)_{i_1j_1} \times X(\fn)_{i_2j_2}),$$
$$\cZ(\fn, g)_{\ul i\ul j}=\cZ(\fn, g) \cap (X(\fn)_{i_1j_1} \times X(\fn)_{i_2j_2}).$$

Let $X:=X(\fn) \times X(\fn)$ and $Y:= X(1) \times X(1)$.  One can see by definition that $X(1)= X(\fn)/ (\cK/\bF_q^* \cK(\fn))$. By the map 
$$ X(\fn) \longrightarrow X(\fn) / (\cK/\bF_q^*\cK(\fn))$$
\noindent one can define the map $\theta: X \longrightarrow Y$. We denote by $\theta_{\ul i\ul j}$ for $\ul i=(i_1, i_2)$ and $\ul j=(j_1, j_2)$ the morphism $\theta$ on the irreducible components $(X(\fn) \times X(\fn))_{\ul i\ul j} \longrightarrow (X(1) \times X(1))_{\ul i}$, i.e, 
$$\theta_{\ul i\ul j}: X(\fn)_{i_1j_1} \times X(\fn)_{i_2j_2} \longrightarrow X(1)_{i_1} \times X(1)_{i_2}.$$

We want to remark that in Section \ref{sec: Briney}, the numerical components are defined for irreducible varieties. Let $\cZ$ be a cycle of $X(1) \times X(1)$. Write $\cZ=\sum_{\ul i} \cZ_{\ul i}$. 
\noindent Let $\cZ_{\ul i\ul j}$  be a component of $\theta_{\ul i \ul j}^{-1}(\cZ_{\ul i})$. We define
$$d_{\ul i\ul j}(\cZ_{ \ul i}):=[\cZ_{\ul i\ul j}:\cZ_{\ul i}] : \mbox{the degree of } \cZ_{\ul i\ul j} \mbox{ over } \cZ_{\ul i},$$
$$d_{\ul i\ul j}^s(\cZ_{ \ul i}):=[\cZ_{\ul i\ul j}:\cZ_{\ul i}]_s : \mbox{the separable degree of } \cZ_{\ul i\ul j} \mbox{ over } \cZ_{\ul i}$$
\noindent and 
$$d_{\ul i\ul j}^i(\cZ_{ \ul i}):=[\cZ_{\ul i \ul j}:\cZ_{\ul i}]_i : \mbox{the inseparable degree of } \cZ_{\ul i\ul j} \mbox{ over } \cZ_{\ul i}.$$
In particular, one has
\[
\sum_{\ul j} d_{\ul i \ul j}(\cZ_{\ul i}) = [X(\fn): X(1)] = [\cK: \bF_q^* \cK(\fn)] \quad \text{ for every ${\ul i}$.}
\]

Unlike Briney, we used the upper script $s$ and $i$ to point out the separable and inseparable degrees to avoid confusion with the irreducible component indices $\ul i\ul j$.

Note that in our case, we have the ``Galois group'' $\fg=\cK/\bF_q^*\cK(\fn) \times \cK/\bF_q^*\cK(\fn)$ and so 
\[
[X(\fn) \times X(\fn) : X(1) \times X(1)]=[\cK:\bF_q^*\cK(\fn)]^2.
\]
In particular: 

 let $\fg_{\ul i \ul j}$ be the Galois group associated with the irreducible component $(X(\fn)\times X(\fn))_{\ul i\ul j}\rightarrow (X(1) \times X(1))_{\ul i}$, of the covering $X(\fn)\times X(\fn)\rightarrow X(1)\times X(1)$. Then $\fg_{\ul i \ul j}$ is a subgroup of $\fg$. 
Also:

\begin{lemma}\label{lem:ell_ij}
    We have $\ell_{\ul i\ul j}(\cZ_{g, i})=1$.
\end{lemma}

\begin{proof}

Given $\sigma \in \fg_{\ul i \ul j}^i(\cZ(\fn,g)_{\ul i \ul j})$ represented by $(\kappa_1,\kappa_2)$ for $\kappa_1,\kappa_2 \in \cK$, one has that for every $x = ([z,b],[z,bg^{-1}])$ in $ \cZ(\fn,g)_{\ul i \ul j}$,
\[
([z,b\kappa_1]_{\fn},[z,bg^{-1}\kappa_2]_{\fn}) = x^{\sigma} = x = ([z,b]_{\fn},[z,bg^{-1}]_{\fn}).
\]
As this equality holds for every $z \in \Omega_r$, we must get $\kappa_1,\kappa_2 \in \bF_{q}^*\cK(\fn)$, which means that $\sigma =1$.
Hence we get $[\fg_{\ul i\ul j}^i(\cZ(\fn,g)_{\ul i\ul j}):1]=1$.

Moreover, the fact that $d_{\ul i\ul j}^i(\cZ_{g,\ul i})=1$ follows from the fact that 
$X(\fn,g)_{ij}\rightarrow X_{g,i}$ is a Galois covering.
Therefore, 
$$\ell_{\ul i\ul j}(\cZ_{g, i})=\frac{[\fg_{\ul i\ul j}^i(\cZ(\fn,g)_{\ul i\ul j}):1]}{d_{\ul i\ul j}^i(\cZ_{g, \ul i})}=1.$$

\end{proof}

Now, assume that ${\rm ord}_v(\fa)\leq {\rm ord}_v(\fn)$ for every $v \notin {\bf Ram}\cup \{\infty\}$, which implies that $g \cK(\fn) g^{-1}, g^{-1}\cK(\fn) g \subset \cK$ as ${\rm Nr}(g)\cO^{\infty}\cap k = \fa$. Put 
\[
\cH_{\fa}:= \{g \in D^*(\bA_f)\cap \cD^{\infty}\mid {\rm Nr}(g)\cO^{\infty} \cap k = \fa\}.
\]

We can rewrite 
$  \sum_{\tilde{g} \in \cK(\fn) \backslash \cH_{\fa} / \cK(\fn)}
\cZ(\fn,g)\cdot \cZ(\fn) $ as

\begin{align*}
\sum_{\tilde{g} \in \cK(\fn) \backslash \cH_{\fa} / \cK(\fn)}
\cZ(\fn,g)\cdot \cZ(\fn) &= \sum_{\tilde{g} \in \cK(\fn) \backslash \cH_{\fa} / \cK(\fn)}
\sum_{\ul i, \ul j}
\cZ(\fn,g)_{\ul i \ul j}\cdot \cZ(\fn)_{\ul i \ul j}  \\
&\hspace{-2cm} =
\sum_{\bar{g} \in \cK \backslash \cE_{\fa} / \cK}
\sum_{(\kappa_1,\kappa_2) \in (\cK/\cK(n))^2/H_g} \sum_{\ul i, \ul j}\cZ(\fn,\kappa_1^{-1} g \kappa_2)_{\ul i \ul j} \cdot \cZ(\fn)_{\ul i \ul j},
\end{align*}
where $H_g = \{(\kappa, g^{-1} \kappa g)\mid \kappa \in \cK_g/\cK(\fn,g)\}$.
Note that $\cZ(\fn)_{\ul i \ul j}$ is nonempty if and only if $\ul i = (i,i)$ and $\ul j = (j,j)$.
Thus we do not need to take all pairs $(\kappa_1,\kappa_2)$ into account. More precisely, for every open compact subgroup $U$ of $D^*(\bA_f)$, let $U^{\flat}:=\{g \in U\mid {\rm Nr}(g) \in \bF_q^*\}$.
Then
\begin{eqnarray*}
& & \sum_{(\kappa_1,\kappa_2) \in (\cK/\cK(n))^2/H_g} \sum_{\ul i, \ul j}\cZ(\fn,\kappa_1^{-1} g \kappa_2)_{\ul i \ul j} \cdot \cZ(\fn)_{\ul i \ul j} \\
&=&
(q-1)\cdot \sum_{(\kappa_1,\kappa_2) \in (\cK^{\flat}/\bF_q^*\cK(n)^{\flat})^2/H_g^{\flat}} \sum_{\ul i, \ul j}\cZ(\fn,\kappa_1^{-1} g \kappa_2)_{\ul i \ul j} \cdot \cZ(\fn)_{\ul i \ul j},
\end{eqnarray*}
where $H_g^{\flat}:=\{(\kappa,g^{-1}\kappa g)\mid \kappa \in \cK_g^{\flat}/\bF_q^*\cK(\fn,g)^{\flat}\}$.
Moreover:

\begin{lemma}
    We have \[
\theta_{\ul i \ul j}^z(\cZ_{g,\ul i})=\sum_{(\kappa_1,\kappa_2) \in (\cK^{\flat}/\bF_q^*\cK(\fn)^{\flat})^2/H^{\flat}_g}
\cZ(\fn,\kappa_1^{-1} g \kappa_2)_{\ul i \ul j}. 
\]
\end{lemma}

\begin{proof}

By definition one has $$\theta_{\ul i\ul j}^z(\cZ_{g, \ul i})=\ell_{\ul i \ul j}(\cZ_{g, \ul i})\cZ_{g, \ul i}^*$$
\noindent where $\cZ_{g, \ul i}^*$
is the sum of the irreducible components of 
$\theta_{\ul i\ul j}^{-1}(\cZ_{g, \ul i})$. 
Note that the Galois group of the covering map $(X(\fn)\times X(\fn))_{\ul i \ul j}\rightarrow (X(1)\times X(1))_{\ul i}$ can be identified with $(\cK^{\flat}/\bF_q^*\cK(\fn)^{\flat})^2$, and one checks that
\[
\cZ(\fn,g)_{\ul i \ul j}^{(\kappa_1,\kappa_2)} = \cZ(\fn,g)_{\ul i \ul j} \quad \text{if and only if }\quad (\kappa_1,\kappa_2) \in H_g^{\flat}.
\]

Since $\ell_{\ul i\ul j}(\cZ_{g, \ul i})=1$ by Lemma \ref{lem:ell_ij}, we get
$$\theta_{\ul i\ul j}^z(\cZ_{g, \ul i})
=\sum_{(\kappa_1,\kappa_2) \in (\cK^{\flat}/\bF_q^*\cK(\fn)^{\flat})^2/H^{\flat}_g}
\cZ(\fn, g)_{\ul i\ul j}^{(\kappa_1, \kappa_2)}.$$
Therefore, it remains to show $\cZ(\fn, g)_{\ul i\ul j}^{(\kappa_1, \kappa_2)}=\cZ(\fn, \kappa_1^{-1}g\kappa_2)_{\ul i\ul j}$, which follows from the 
identity
\[
\cZ(\fn,g)^{(\kappa_1,\kappa_2)} = \cZ(\fn,\kappa_1^{-1}g\kappa_2), \quad \forall \kappa_1,\kappa_2 \in \cK.
\]

\end{proof}

\begin{remark}
It is known that the self-intersection number $i(\cZ(\fn)\cdot \cZ(\fn)) = \sum_{\ul i \ul j} i (\cZ(\fn)_{\ul i \ul j}\cdot \cZ(\fn)_{\ul i \ul j})$ is equal to $\chi\big(X(\fn)\big)$, the ``Euler--Poincar\'e characteristic of $\cZ(\fn)$''.
We refer the reader to Section~\ref{sec: S-I} for further discussion and the precise formula for $\chi\big(X(\fn)\big)$. 
\end{remark}

\noindent {\it Proof of Theorem~\ref{ex: proj-formula}.}
From the previous discussion with the projection formula,
we then get that
\begin{eqnarray*}
\sum_{\tilde{g} \in \cK(\fn) \backslash \cH_{\fa} / \cK(\fn)}
i(\cZ(\fn,g)\cdot \cZ(\fn))
&=& (q-1) \cdot \sum_{\bar{g} \in \cK\backslash \cH_{\fa}/ \cK} \sum_{\ul i , \ul j} i\big(\theta^z_{\ul i \ul j}(\cZ_{g,\ul i}) \cdot \cZ(\fn)_{\ul i \ul j}\big) \\
&=&(q-1)\cdot [\cK: \bF_q^*\cK(\fn)] \cdot \sum_{\bar{g} \in \cK\backslash \cH_{\fa}/  \cK}  \sum_{\ul i} i(\cZ_{g,\ul i} \cdot \cZ_{1,\ul i})\\
&=& [\cK:\cK(\fn)] \cdot i(\cZ\cdot \cZ_{\fa}).
\end{eqnarray*}
Therefore the result holds.
\hfill $\Box$

In the next section, we will determine the transversality of the intersection $\cZ(\fn,g)\cdot \cZ(\fn)$ when $\cZ(\fn,g)\neq \cZ(\fn)$.

\section{Transversal intersection}\label{sec: tran-int}

We start with recalling the following lemma:

\begin{lemma}(\cite[Chapter 3, Exercise 6.7]{GW})\label{lem-int}

Let $F$ be a field, let $X$ be a $k$-scheme, and let $Y_1$ and $Y_2$ be closed
subschemes of $X$ and let $Y_1 \cap Y_2$ be their schematic intersection. Let $x \in  (Y_1 \cap Y_2)(F)$ be an $F$-valued point and assume that $X$, $Y_1$, and $Y_2$ are smooth at $x$ over $F$ of relative dimension $d, d - c_1$ and $d - c_2$, respectively. The following assertions are equivalent.
\begin{enumerate}

\item $Y_1 \cap  Y_2$ is smooth of relative dimension $d - (c_1 + c_2)$.

\item $T_x Y_1 + T_x Y_2 = T_x X$.

\end{enumerate}
Here $T_x Y_i$ is the tangent space of $Y_i$ at $x$ for $i=1,2$, regarding as subspaces of $T_x X$, the tangent space of $X$ at $x$.

If these equivalent conditions are satisfied, we say that $Y_1$ and $Y_2$ intersect transversally.
\end{lemma}

Recall that we defined a morphism 
$$\pi_{\fn, g}: X(\fn, g) \longrightarrow X(\fn) \times X(\fn)$$
\noindent for every $g \in D^*(\bA_f)$ in Section \ref{sec-cycles}.
Put 
\[
\bX(\fn,g):=\pi_{\fn, g}(X(\fn, g))
\quad \text{ and } \quad 
\bX(\fn):= \bX(\fn,1).
\]
As $\bX(\fn,g)$ and $\bX(\fn)$ are both smooth of dimension $r-1$ in $X(\fn)\times X(\fn)$ (which has dimension $2r-2$), showing the transversality of the intersection $\bX(\fn) \cap \bX(\fn, g)$ reduces to verify that the intersection of the corresponding ``tangent spaces'' at each intersection point is trivial.
To proceed,
we will now consider our objects as rigid analytic spaces, and study the corresponding tangent spaces via their rigid analytic uniformization.\\

Recall the following identification in Proposition \ref{prop: rigid}: 
$$X(\fn) \times X(\fn) = \Big(\coprod_i X(\fn)_i\Big)\times \Big(\coprod_j X(\fn)_j\Big) \simeq \Big(\coprod_i \Gamma_i \backslash \Omega_r\Big) \times \Big(\coprod_j \Gamma_j \backslash \Omega_r\Big).$$
Put
$\bH_\gamma:=\{ (z, \gamma z) \mid z \in \Omega_r \}$
for every $\gamma \in D_\infty^* \cong \GL_r(k_\infty)$.
Then 
\[
\bX(\fn) \simeq \coprod_{i} \Gamma_i \backslash \bH_1,
\quad \text{ where $\Gamma_i$ acts diagonally on $\bH_1$}.
\]
Similarly, there exist elements $\gamma_{ij} \in D^*$ and arithmetic subgroups $\Gamma_{\fn,\gamma_{ij}} \subset D_\infty^*$ such that
$$\bX(\fn, g) = \coprod_{i,j} \bX(\fn,g)_{ij}, 
$$
where
$$
\bX(\fn,g)_{ij} := \bX(\fn,g) \cap \big(X(\fn)_i \times X(\fn)_j\big)
\simeq
\Gamma_{\fn, \gamma_{ij}} \backslash \bH_{\gamma_{ij}} \subset (\Gamma_i\times \Gamma_j)\backslash (\Omega_r \times \Omega_r).
$$
Here $\Gamma_{\fn, \gamma_{ij}}$ acts diagonally on $\bH_{\gamma_{ij}}$ as well.
Let $\alpha \in \bX(\fn, g) \cap \bX(\fn)$. Since the covering $\Omega_r \longrightarrow \Gamma_i \backslash \Omega_r$ is \'etale and the intersection behavior is a local property, it is sufficient to lift $\alpha$ to a point (still denoted by $\alpha$ by abuse of notation) in 
$\bH_{\gamma} \cap \bH_1 \subset \Omega_r \times \Omega_r$ 
for some nonconstant $\gamma \in D^*$.

\begin{lemma}
Suppose $r$ is a prime number distinct from the characteristic of $k$. Given $\gamma \in D^* \subset D_\infty^*\cong \GL_r(k_\infty)$ such that $\bH_{\gamma} \neq \bH_{1}$, suppose there exists $\alpha \in \bH_{\gamma} \cap \bH_{1} \subset \Omega_r\times \Omega_r$.
Let $T_{\alpha} \Omega_r^2 \cong \bC_{\infty}^{r-1}\times \bC_{\infty}^{r-1}$ be the tangent space of $\Omega_r \times \Omega_r$ at $\alpha$, and
$T_{\alpha} \bH_{\gamma}$ be the
tangent space of $\bH_{\gamma}$ at $\alpha$ (regarded as a
subspace of $T_{ \alpha}\Omega_r^2$
via the inclusion $\bH_{\gamma} \subset \Omega_r \times \Omega_r$).
Then $T_{\alpha}\bH_{\gamma}  \cap T_{\alpha} \bH_1 = \{0\}$.
\end{lemma}

\begin{proof}
For simplicity, we may identify $\Omega_r$ with
\[
\left\{\begin{pmatrix} z_1 \\ \vdots \\ z_{r-1}\end{pmatrix} \in \bC_{\infty}^{r-1} \ \Bigg|\ c_1z_1+\cdots + c_{r-1}z_{r-1}+c_r \neq 0,\ \forall 0\neq (c_1,...,c_r) \in k_\infty^r\right\}.
\]
Given 
$$z= \begin{pmatrix} z_1 \\ \vdots \\ z_{r-1}\end{pmatrix}  \in \Omega_r 
\quad \mbox{ and } \quad 
\gamma =\begin{pmatrix}
    a_{11} & \cdots & a_{1r} \\
    \vdots & \ddots & \vdots \\
    a_{r1} & \cdots & a_{rr}   
\end{pmatrix} \in \GL_r(k_\infty),$$
write
\[
\begin{pmatrix}
    w_1 \\
    \vdots \\
    w_{r-1} \\
    w_r
\end{pmatrix}
:=
\begin{pmatrix}
a_{11} & \cdots & a_{1r} \\
    \vdots & \ddots & \vdots \\
    a_{r1} & \cdots & a_{rr}   
\end{pmatrix} 
\begin{pmatrix}
     z_1 \\
    \vdots \\
    z_{r-1}\\
    1
\end{pmatrix},
\quad \text{which says} \quad 
w:= \gamma z = 
\begin{pmatrix}
    w_1/w_r \\ \vdots \\ w_{r-1}/w_r 
\end{pmatrix} \in \Omega_r.
\]
Regarding $w_1,...,w_r$ as functions in $z_1,...,z_{r-1}$, let
$\partial_j w$ be the partial derivative of $w$ with respect to $z_j$, i.e, 
\begin{equation}\label{eqn: partial-w}
\partial_j w = \begin{pmatrix}
    \frac{\partial}{\partial z_j} \frac{w_1}{w_r} \\
    \vdots \\
    \frac{\partial}{\partial z_j}  \frac{w_{r-1}}{w_r}
\end{pmatrix}
\quad \text{ and }
\quad 
\frac{\partial}{ \partial z_j} \frac{w_i}{w_r} = w_r^{-1} (a_{ij} - a_{rj} \frac{w_i}{w_r}), \quad 1\leq i,j \leq r-1.
\end{equation}
Then for every $\alpha = (z^o,\gamma z^o) \in \bH_\gamma$, the tangent space $T_{\alpha}\bH_{\gamma}$ is spanned by the vectors
\[
\big(\begin{pmatrix}1 \\ \vdots \\ 0\end{pmatrix}, (\partial_1 w)(z^o)\big),..., 
\big(\begin{pmatrix}0 \\ \vdots \\ 1\end{pmatrix}, (\partial_{r-1} w)(z^o)\big) \quad \in \bC_{\infty}^{r-1} \times \bC_{\infty}^{r-1} \cong T_{\Omega_r^2,\alpha}.
\]
In particular,
the tangent space $T_{(z^o,z^o)}\bH_{1}$ is spanned by the vectors
\[
\big(\begin{pmatrix}1 \\ \vdots \\ 0\end{pmatrix}, \begin{pmatrix}1 \\ \vdots \\ 0\end{pmatrix},..., 
\big(\begin{pmatrix}0 \\ \vdots \\ 1\end{pmatrix}, \begin{pmatrix}0 \\ \vdots \\ 1\end{pmatrix}\big) \quad \in \bC_{\infty}^{r-1} \times \bC_{\infty}^{r-1} \cong T_{\Omega_r^2,(z^o,z^o)}.
\]

Suppose that $\alpha \in \bH_{\gamma}\cap \bH_1$, i.e.\ $\gamma z^o = z^o$.
Write
\[
z^o = \begin{pmatrix} z^o_1 \\ \vdots \\ z^o_{r-1} \end{pmatrix}
\quad \text{and} \quad 
w_i^o := w_i(z^o), \ 1\leq i \leq r-1.
\]
Then $\gamma z^o = z^o$ implies that $w_i^o/w_r^o = z_i^o$ for $1\leq i\leq r-1$, and 
\begin{equation}\label{eq: Az}
    A \begin{pmatrix}
    z_1^o \\
    \vdots \\
    z_{r-1}^o \\
    1
\end{pmatrix} = 0, \quad \text{ where } A:= 
    \begin{pmatrix}
    a_{11}-w_r^o & a_{12} & \cdots & a_{1r} \\
    a_{21} & a_{22}-w_r^o & \cdots & a_{2r} \\
    \vdots & \vdots  & \ddots & \vdots \\
    a_{r1} & a_{r2} & \cdots & a_{rr}-w_r^o
\end{pmatrix} .
\end{equation}

On the other hand, given $y \in T_{\alpha} \bH_{\gamma} \cap T_{\alpha} \bH_{1}$,
we may write $y= (x,x)$ with $x \in \bC_{\infty}^{r-1}$, and the condition $y=(x,x) \in T_{\alpha} \bH_{\gamma}$ implies that 
\begin{equation*} \label{eq: parder}
x = \begin{pmatrix}
    x_1\\
    \vdots \\
    x_{r-1}
\end{pmatrix} = x_1 (\partial_1 w)(z^o) + \cdots + x_{r-1} (\partial_{r-1} w)(z^o).
\end{equation*}
By \eqref{eqn: partial-w} we get
\begin{equation*}
    \begin{pmatrix}
        x_1 \\
        \vdots \\
        x_{r-1}
    \end{pmatrix} = (w_r^o)^{-1} \begin{pmatrix}
        a_{11}-a_{r1}z_1^o & \cdots & a_{r-1, 1}-a_{r1}z_{r-1}^o \\
        & \vdots & \\
        a_{1, r-1}-a_{r, r-1}z_1^o & \cdots & a_{r-1, r-1}-a_{r, r-1} z_{r-1}^o
    \end{pmatrix} \begin{pmatrix}
        x_1 \\
        \vdots \\
        x_{r-1}
    \end{pmatrix}
\end{equation*}
which is equivalent to
\begin{equation}\label{eq: LA}
B\begin{pmatrix}
        x_1 \\
        \vdots \\
        x_{r-1}
    \end{pmatrix} = 0, \quad \text{ where } B:=
    \begin{pmatrix}
        a_{11}-a_{r1}z_1^o -w_r^o & \cdots & a_{r-1, 1}-a_{r1}z_{r-1}^o \\
       \vdots & \ddots & \vdots \\
        a_{1, r-1}-a_{r, r-1}z_1^o & \cdots & a_{r-1, r-1}-a_{r, r-1} z_{r-1}^o-w_r^o
    \end{pmatrix}.
\end{equation}
By suitable row operations on the following matrix 
\[
B':=\begin{pmatrix}
    a_{11}-w_r^o & a_{21} &  & \cdots & a_{r1} \\
    a_{12} & a_{22}-w_r^o &  &\cdots & a_{r2} \\
    \vdots & \vdots  &  \ddots &  & \vdots \\
    a_{1,r-1} & a_{2,r-1} &  \cdots & a_{r-1,r-1}-w_r^o & a_{r,r-1}\\
    z_1^o & z_2^o &  \cdots & z_{r-1}^o & 1
\end{pmatrix},
\]
we can see that $\det(B')=\det(B)$.

Now, if $y=(x,x)\neq 0$, i.e.\ $x \neq 0$, then $\det(B') = \det(B)=0$.
As the field $k(\gamma)$ is separable of degree $r$ over $k$ under our assumption, the eigenspace of $A$ corresponding to the eigenvalue $0$ (with multiplicity one) is spanned by $(z_1^o,...,z_{r-1}^o,1)^t \in \bC_{\infty}^r$.
Hence $\det(B')=0$ implies that
the nonzero column vector $(z_1^o,...,z_{r-1}^o,1)^t \in \bC_{\infty}^r$ lies in the range of $A$, and also in the null space of $A$ by \eqref{eq: Az}.
However, the diagonalizability of $A$ in $M_r(\bC_{\infty})$ assures that the intersection of the range of $A$ and the null space of $A$ must be trivial, which is a contradiction.
Therefore $y = 0$, i.e.~the intersection of $T_{\alpha} \bH_{\gamma}$ and $T_{\alpha} \bH_1$ is trivial.
\end{proof}

Consequently, we have that:

\begin{corollary}
Suppose $r$ is a prime number distinct from the characteristic of $k$. 
When $\bX(\fn,g)\neq \bX(\fn)$, the intersection $\bX(\fn, g) \cap \bX(\fn)$ is transversal.
\end{corollary}

\section{Counting the intersection numbers of Hecke correspondences}
\label{sec: count}

From the previous section, we know that when $r$ is a prime number distinct from the characteristic of $k$, the intersection of $\bX(\fn)$ and $\bX(\fn,g)$ is always transversal when they are distinct.
In this case, the intersection number $i(\cZ(\fn)\cdot \cZ(\fn,g))$ is simply equal to the cardinality of their intersection points.
In what follows, we shall count the intersection points in question by using ``optimal embeddings.''

We first release our condition on $r$, (i.e.\ $r$ is just a positive integer). Recall the following uniformization of $X(\mathfrak{n},g)$:
\[
X(\mathfrak{n},g)(\mathbb{C}_\infty) \cong 
D^* \backslash \Omega_r \times D^*(\bA_f) / \cK(\fn, g).
\]

Every $\bC_\infty$-valued point of $X(\mathfrak{n},g)$ corresponds to a class in $D^* \backslash \Omega_r \times D^*(\bA_f) / K(\fn, g)$.
Note that the morphism $\pi_{\fn,g}: X(\fn,g)\rightarrow X(\fn)\times X(\fn)$ gives an isomorphism between $X(\fn,g)$ and $\bX(\fn,g)$.
Then:

\begin{lemma}\label{lem: emb-1}
The set $\pi_{\fn,g}^{-1}\big(\mathbb{X}(\mathfrak{n})\cap \mathbb{X}(\mathfrak{n},g)\big)(\bC_\infty)$ can be identified with 
\[
\left\{
[z,b]_{\fn,g} \in D^* \backslash \Omega_r \times D^*(\bA_f) / \cK(\fn, g)\ \Bigg|\ 
\begin{tabular}{l}
\text{there exists  $\gamma \in D^*$ so that} \\
\text{$\gamma \cdot z = z$ and $b^{-1} \gamma b \in \cK(\mathfrak{n})g$}
\end{tabular}
\right\}.
\]
\end{lemma}

\begin{proof}
Let $x \in X(\fn,g)(\bC_\infty)$.
From the unifomization of $X(\fn,g)$, we may identify $x$ with a class $[z,b]_{\fn,g} \in D^* \backslash \Omega_r \times D^*(\bA_f) / \cK(\fn, g)$.
Then $[z,b]_\fn$ (resp.\ $[z,bg^{-1}]_\fn$) is the class in $D^* \backslash \Omega_r \times D^*(\bA^{\infty}) / K(\fn)$ corresponding to $\pi_{\fn,g,1}(x)$ (resp.\ $\pi_{\fn,g,2}(x)$).
Suppose $\pi_{\fn,g}(x) \in \bX(\fn)$, which is equivalent to
\[
[z,b]_{\fn}=[z,bg^{-1}]_{\fn} \in D^* \backslash \Omega_r \times D^*(\bA_f) / \cK(\fn).
\]
This means that there exists $\gamma \in D^*$ and $\kappa \in \cK(\fn)$ so that $\gamma z = z$ and $\gamma b g^{-1} = b \kappa$, which says that 
\[
b^{-1} \gamma b = \kappa g \in \cK(\fn)g.
\]
Conversely, suppose there exists $\gamma \in D^*$ so that $\gamma \cdot z = z$ and $b^{-1}\gamma b \in \cK(\fn)g$.
Write $b^{-1}\gamma b = \kappa g$ where $\kappa \in \cK(\fn)$.
Then
\[
[z,b] = [\gamma^{-1}z,\gamma^{-1}b\kappa] = [z,bg^{-1}] \in D^* \backslash \Omega_r \times D^*(\bA_f) / \cK(\fn),
\]
which implies that $\pi_{\fn,g}(x) \in \bX(\fn)$.
This completes the proof.
\end{proof}

Set
\[
\mathcal{S}(\fn,g):=
\big\{(\gamma,z,b) \in D^*\times \Omega_r \times D^*(\bA_f) \mid \gamma \cdot z = z,\  b^{-1}\gamma b \in \cK(\fn)g\big\},
\]
which is equipped with a left action of $D^*$ and a right action of $K(\fn,g)$ defined below: for every $(\gamma,z,b) \in \mathcal{S}(\fn,g)$, $\gamma_0 \in D^*$, and $\kappa \in K(\fn,g)$,
\[
\gamma_0 \cdot (\gamma,z,b) \cdot \kappa := (\gamma_0 \gamma \gamma_0^{-1}, \gamma_0 \cdot z, \gamma_0 b \kappa).
\]
By Lemma~\ref{lem: emb-1}, we have a natural surjective map from $\mathcal{S}(\fn,g)$ to $\pi_{\fn,g}^{-1}\big(\bX(\fn)\cap \bX(\fn,g)\big)(\bC_\infty)$ sending $(\gamma,z,b)$ to (the point corresponding to) $[z,b]_{\fn,g}$.
Moreover:

\begin{lemma}\label{lem: emb-2}
The above map induces a bijection between
$D^*\backslash \mathcal{S}(\fn,g)/\cK(\fn,g)$ and $\pi_{\fn,g}^{-1}\big(\bX(\fn)\cap \bX(\fn,g)\big)(\bC_\infty)$.
\end{lemma}

\begin{proof}
Given $(\gamma,z,b),(\gamma',z',b') \in \mathcal{S}(\fn,g)$, suppose 
\[
[z,b]_{\fn,g} = [z',b']_{\fn,g} \quad \in X(\fn,g),
\]
which implies that there exists $\gamma_0 \in D^*$ and $\kappa_0 \in \cK(\fn,g)$ so that
$z' = \gamma_0 z$ and $b' = \gamma_0 b \kappa_0$.
Thus
\[
\gamma_0^{-1}\cdot(\gamma',z',b')\cdot \kappa_0^{-1} = (\gamma_0^{-1}\gamma'\gamma_0,z,b) \quad \in \cS(\fn,g).
\]
Put $\gamma_1 = \gamma_0^{-1} \gamma' \gamma_0 \in D^*$, which satisfies that $\gamma_1 \cdot z = z$ and $b^{-1} \gamma_1 b \in K(\fn)g$.

Note that by \cite[Lemma 4.6]{Pap}, one has that $K_z := \{ \gamma \in D^*\mid \gamma \cdot z=z\}\cup \{0\}$ is a subfield of $D$ which is imaginary with respect to $\infty$ (i.e.\ $\infty$ is non-split in $K_z$).
Hence
\[
\gamma_2:= \gamma_1 \gamma^{-1} \in K_z^* \quad \text{with} \quad b^{-1}\gamma_2 b \in \cK(\fn).
\]
This says in particular that $\gamma_2$ is a unit in $O_{K_z}$, the integral closure of $A$ in $K_z$, and $b^{-1}\gamma_2 b \equiv 1 \bmod \fn$.
Since $K_z$ is imaginary and the degree $[K_z:k]$ divides $r$, the units of $O_{K_z}$ must be contained in a finite field with $q^r$ elements.
The condition $b^{-1}\gamma_2 b \equiv 1 \bmod \fp$ then implies that $\gamma_2 = 1$,
whence
\[
\gamma_0^{-1} \cdot (\gamma',z',b')\kappa_0^{-1} = (\gamma,z,b).
\]
Therefore $(\gamma,z,b)$ and $(\gamma',z',b')$ represent the same class in $D^*\backslash \mathcal{S}(\fn,g)/\cK(\fn,g)$, and the proof is complete.
\end{proof}

\begin{corollary}\label{cor: sel-int}
Given $g \in D^*(\bA_f)$, one has that $\bX(\fn) = \bX(\fn,g)$ if and only if $g \in k^* \cdot \cK(\fn)$.
\end{corollary}

\begin{proof}
It is clear that $\bX(\fn) = \bX(\fn,g)$ when $g \in k^* \cdot \cK(\fn)$.
Conversely, suppose $\bX(\fn) = \bX(\fn,g)$.
By Lemma~\ref{lem: emb-2}, we have that for every $z \in \Omega_r$, there exists $\gamma \in D^*$ such that $\gamma \cdot z = z$ and $b^{-1} \gamma b \in \cK(\fn)g$.
When taking $z$ with algebraically independent coordinates over $k$, we must have that $\gamma \in k^*$,
whence $g = \gamma \kappa \in k^* \cK(\fn)$ as desired.
\end{proof}

\subsection[The case when X(n) neq X(n,g)]{The case when $\bX(\fn) \neq \bX(\fn,g)$}

Let 

\[
\cE(\fn,g) :=\big\{ (\gamma,b) \in D^* \times D^*(\bA_f)\ \big|\ k(\gamma) \text{ is imaginary over $k$, } b^{-1} \gamma b \in  \cK(\fn)g\big\}.
\]
Then we have a natural map from $\phi:\cS(\fn,g)\rightarrow \cE(\fn,g)$ defined by $\phi(\gamma,z,b)=(\gamma,b)$ for every triple $(\gamma,z,b) \in \cS(\fn,g)$.
On the other hand, the fiber of each $(\gamma,b)$ in $\cE(\fn,g)$ under $\phi$ can be identified with the set of fix points of $\gamma$ on $\Omega_r$, which is always non-empty.
Therefore $\phi$ is surjective.
Moreover, the following lemma holds:

\begin{lemma}
The map
\[
\bar{\phi}: D^*\backslash \cS(\fn,g)/\cK(\fn,g) \longrightarrow D^*\backslash \cE(\fn,g)/\cK(\fn,g) 
\]
induced by $\phi$ is surjective.
Moreover, 
suppose $\bX(\fn) \neq \bX(\fn,g)$ and $r$ is a prime number.
For each class $[\gamma,b] \in D^*\backslash \cE(\fn,g)/K(\fn,g)$, the cardinality of $\bar{\phi}^{-1}([\gamma,b])$ is $r$ $($resp.\ $1)$ if $k(\gamma)$ is separable $($resp.\ purely inseparable$)$ over $k$.
\end{lemma}

\begin{proof}
The surjectivity is clear.
Suppose $\bX(\fn) \neq \bX(\fn,g)$ and $r$ is a prime number.
Let $(\gamma,b) \in \cE(\fn,g)$.
By assumption one has that
$k \neq k(\gamma) \subset D$ is a maximal subfield of $D$ which is imaginary.
Thus there are $r$ (resp.\ $1$) fixed point(s) of $\gamma$ on $\Omega_r$, say $z_1,...,z_r$ (resp.\ $z_0$), corresponding to distinct eigenvalues (resp.\ the unique eigenvalue) of $\gamma$.
Hence
\[
\phi^{-1}(\gamma,b) = 
\begin{cases}
\{(\gamma,z_1,b),\cdots (\gamma,z_r,b)\}, & \text{ if $k(\gamma)/k$ is separable},\\
\{(\gamma,z_0,b)\}, & \text{ if $k(\gamma)/k$ is purely inseparable.}
\end{cases}
\]

It suffices to verify that $(\gamma,z_1,b),...,(\gamma,z_r,b)$ represents distinct classes in the space $D^*\backslash \cS(\fn,g)/\cK(\fn,g)$ when $k(\gamma)/k$ is separable.
To show this,
suppose there exists $\gamma_0 \in D^*$ and $\kappa_0 \in \cK(\fn,g)$ such that
\[
(\gamma,z_i,b)  = \gamma_0\cdot (\gamma,z_j,b) \cdot \kappa_0 = (\gamma_0 \gamma \gamma_0^{-1}, \gamma_0 \cdot z_j, \gamma_0 b \kappa_0).
\]
Since $k(\gamma)$ is a maximal subfield of $D$, the centralizer of $k(\gamma)$ is itself.
Hence $\gamma_0 \in k(\gamma)^*$ and
\[
b^{-1}\gamma_0 b = \kappa_0^{-1} \in \cK(\fn,g).
\]
By our assumption on $\fn$, we can apply the same argument in the proof of Lemma~\ref{lem: emb-2} to get $\gamma_0 = 1$, which implies $z_i=z_j$ and $\kappa_0 = 1$.
\end{proof}

\begin{corollary}\label{cor: emb}
Suppose $X(\fn) \neq X(\fn,g)$ and $r$ is a prime number distinct from the characteristic of $k$.
The cardinality of $\Delta_\fn^{-1}\big(\bX(\fn)\cap \bX(\fn,g)\big)(\bC_\infty)$ is equal to 
\[
r \cdot \#\left(D^*\backslash \cE(\fn,g)/\cK(\fn,g)\right).
\]
\end{corollary}

\begin{remark}\label{rem: ind-k} 
Given $\kappa \in \cK(\fn)$, the map from $\cE(\fn,g)$ to $\cE(\fn,g\kappa)$ sending $(\gamma,b)$ to $(\gamma,b\kappa)$ for every $(\gamma,b) \in \cE(\fn,g)$ is bijective.
Thus we have the following induced bijection
\[
D^*\backslash \cE(\fn,g) /\cK(\fn,g) \stackrel{\sim}{\longrightarrow} D^*\backslash \cE(\fn,g\kappa)/\cK(\fn,g\kappa).
\]
\end{remark}

\subsection[Self-intersection of X(n)]{Self-intersection of $\bX(\fn)$}\label{sec: S-I}
${}$
\indent When $\bX(\fn) = \bX(\fn,g)$ (i.e.~$g \in k^* \cdot \cK(\fn)$ by Corollary~\ref{cor: sel-int}), the self-intersection number of $\bX(\fn)$ in $X(\fn)\times X(\fn)$ is equal to the {\it Euler--Poincar\'e characteristic of $X(\fn)$}, denoted by $\chi(X(\fn))$, which is the degree of the top Chern class of the tangent bundle of $X(\fn)$ (see \cite[Example 8.1.12]{Ful}).
Recall the following identification
\[
X(\fn)(\bC_\infty) \cong D^* \backslash \Omega_r \times D^*(\bA_f) / \cK(\fn)
\cong \coprod_{[b] \in D^* \backslash D^*(\bA_f) / \cK(\fn)} \Gamma_b(\fn) \backslash \Omega_r,
\]
where $\Gamma_b(\fn) := D^* \cap b \cK(\fn) b^{-1}$ for every $b \in D^*(\bA_f)$.
Identifying $D_\infty = D\otimes_k k_\infty$ with $\bM_r(k_\infty)$, we may regard $\Gamma_b$ as a discrete and co-compact torsion free subgroup of $\text{GL}_r(k_\infty)$.
Therefore by Kurihara's analogue of Hirzebruch proportionality (see \cite[Theorem 2.2.8]{Ku}), we obtain:
\begin{align}\label{eqn: E-P-C}
\chi\big(X(\fn)\big) &= \sum_{[b] \in D^* \backslash D^*(\bA_f) / \cK(\fn)} \mu_\infty\big(\Gamma_b(\fn) \backslash \text{GL}_r(k_\infty)/k_\infty^*\big).
\end{align}
Here $\mu_\infty$ is the {\it Euler--Poincar\'e measure} on $\text{PGL}_r(k_\infty)$ introduced by Serre (see~\cite{Serre}), i.e. 
\[
\mu_\infty\big(\text{PGL}_r(O_\infty)\big) = \prod_{i=1}^{r-1}(1-q_\infty^i),
\]
and $q_{\infty}$ is the cardinality of the residue field at $\infty$.
Extending $\mu_\infty$ to the Haar measure $\mu_\bA$ on $D^*(\bA)/k_\infty^* \cong D^*(\bA_f)\times \text{PGL}_r(k_\infty)$ satisfying that 
\[
\mu_\bA\big(\cK\times\text{PGL}_r(O_\infty)\big) = \mu_\infty\big(\text{PGL}_r(O_\infty)\big)=\prod_{i=1}^{r-1}(1-q_\infty^i),
\]
the equation~\eqref{eqn: E-P-C} leads to the following Gauss--Bonnet-type formula:

\begin{proposition}\label{prop: G-B-formula}
\[
\chi\big(X(\fn)\big)
= [\cK:  \cK(\fn)] 
\cdot \mu_\bA\big(D^*\backslash D^*(\bA)/k_\infty^*\big).
\]
\end{proposition}

\begin{remark}\label{rem: G-B-T}
(1) We may rewrite the above formula as the following form:
\[
\chi\big(X(\fn)\big)
= [\cK:  \cK(\fn)] \cdot \sum_{[b] \in D^*\backslash D^*(\bA_f)/\cK} \mu_\infty\big(\Gamma_b \backslash \text{GL}_r(k_\infty)/k_\infty^*\big),
\]
where
$\Gamma_b:= D^*\cap b\cK b^{-1}$ for every $b \in D^*(\bA_f)$.

(2) Let 
$\mu_{\bA,t}$ be the {\it Tamagawa measure on $D^*(\bA)/\bA^*$} introduced in \cite[\S 3.2]{Weil}.
It is known that (see \cite[ Theorem~3.2.1]{Weil})
\[
\mu_{\bA,t}\big(D^*\backslash D^*(\bA)/\bA^*\big) = r,
\]
and (cf.\ \cite[\S 3.1, p.~32]{Weil} and \cite[\S 5]{W-Y}) 
\[
\mu_{\bA,t}\big(\cK/(\cO^{\infty})^* \times \text{PGL}_r(\cO_{\infty})\big) = \prod_{i=1}^{r-1}\zeta_k(-i)^{-1} \cdot \prod_{v \in {\bf Ram}} \prod_{i=1}^{r-1}\frac{1}{1-q_v^i},
\]
where $q_v$ is the cardinality of the residue field of $v$ for each $v \in |C|$, and $\zeta_k(s)$ is the zeta function of $k$, i.e.
\[
\zeta_k(s) = \prod_{v \in |C|}(1-q_v^{-s})^{-1}, \quad  {\rm Re}(s)>1.
\]
Therefore we get
\begin{eqnarray}\label{eqn: Total-volume}
&&\mu_\bA\big(D^*\backslash D^*(\bA)/k_\infty^*\big) \nonumber \\
&=&
\frac{\#\text{Pic}(A)}{q-1} \cdot \mu_\bA\big(D^*\backslash D^*(\bA)/\bA^*\big) \nonumber \\
&=&
\frac{\#\text{Pic}(A)}{q-1} \cdot \mu_{\bA,t}\big(D^*\backslash D^*(\bA)/\bA^*\big)\cdot
\frac{\mu_\bA\big(\cK/(\cO^{\infty})^* \times \text{PGL}_r(\cO_{\infty})\big)}{\mu_{\bA,t}\big(\cK/(\cO^{\infty})^* \times \text{PGL}_r(\cO_{\infty})\big)} \nonumber \\
&=& 
\frac{\#\text{Pic}(A)}{q-1} \cdot r \cdot 
\prod_{i=1}^{r-1}\zeta_k(-i) \cdot \prod_{v \in {\bf Ram}\cup\{\infty\}} \, \prod_{i=1}^{r-1}(1-q_v^i).
\end{eqnarray}
Consequently, Proposition~\ref{prop: G-B-formula} implies the following formula.

\begin{proposition}\label{prop: self-int}
The self-intersection number of $\bX(\fn)$ in $X(\fn)\times X(\fn)$ is equal to
\[
\chi\big(X(\fn)\big)
=[\cK:  \cK(\fn)] \cdot r \cdot \frac{\#\text{\rm Pic}(A)}{q-1} \cdot \prod_{i=1}^{r-1} \zeta_k(-i)\cdot \prod_{v \in {\bf Ram}\cup\{\infty\}} \prod_{i=1}^{r-1}(1-q_v^i).
\]
\end{proposition}
\end{remark}

\begin{remark}
When $A$ is a polynomial ring over $\bF_q$, the formula for $\chi\big(X(\fn)\big)$ appears in \cite[equation (6.10)]{Pap09}.
\end{remark}

\subsection{Projection to the full-level case}

Now, we assume $r$ to be a prime number distinct from the characteristic of $k$.
Let $\fa$ be a non-zero ideal of $A$.
Set
\[
\cE_\fa := \left\{
(\gamma,b) \in D^* \times D^*(\bA_f)\ \bigg|\ 
\begin{tabular}{cc}
\text{$k(\gamma)$ is imaginary over $k$,}\\
\text{$b^{-1}\gamma b \in \cD^\infty$
and  $\text{Nr}(\gamma) \cdot A = \fa$}
\end{tabular}
\right\}
= \cE_\fa^{\times} \stackrel{\cdot}{\cup}
\cE_\fa^{\circ}
\]
where
\begin{eqnarray*}
\cE_\fa^{\times} &:=& \left\{
(\gamma,b) \in D^* \times D^*(\bA_f)\ \bigg|\ 
\begin{tabular}{cc}
\text{$k(\gamma)/k$ is imaginary of degree $r$,}\\
\text{$b^{-1}\gamma b \in \cD^\infty$
and  $\text{Nr}(\gamma) \cdot A = \fa$}
\end{tabular}
\right\} \\
\text{and}\quad \quad  
\cE_\fa^{\circ}
&:=&
\big\{(\gamma,b) \in D^* \times D^*(\bA_f) \ \big|\
\gamma \in A \text{ and } \gamma^r \cdot A = \fa\big\}.
\end{eqnarray*}

Then 
$\cE_\fa$ (as well as $\cE_\fa^\times$ and $\cE_\fa^\circ$)
is equipped with a left action of $D^*$ and a right action of $\cK$ defined by
\[
\gamma_0 \cdot (\gamma,b) \cdot \kappa := (\gamma_0 \gamma \gamma_0^{-1}, \gamma_0 b \kappa ), \quad \forall \gamma_0 \in D^*, (\gamma,b) \in \cE_\fa, \kappa \in \cK.
\]
Choose a nonzero ideal $\fn$ of $A$ so that $v \nmid \fn$ for every $v \in {\bf Ram}$ and ${\rm ord}_{v}(\fa)\leq {\rm ord}_{v}(\fn)$ for every $v \notin {\bf Ram}\cup \{\infty\}$.
For each $g \in \cD^\infty \cap D^*(\bA_f)$ with $\text{Nr}(g)\cdot \cO^{\infty}\cap k = \fa$, one has $\cE(\fn,g)
\subset \cE_\fa$.
Put
\[
\cE^\times(\fn,g):= \cE(\fn,g)\cap \cE_\fa^\times
\quad \text{ and } \quad 
\cE^\circ(\fn,g):=
\cE(\fn,g)\cap \cE_\fa^\circ.
\]
We point out that 
either $\cE^\times(\fn,g)$ or $\cE^\circ(\fn,g)$ is empty, and
$\cE^\circ(\fn,g)$ is non-empty if and only if $g \in k^*\cdot \cK(\fn)$.
Moreover, let $\fa'$ be the ``prime-to-{\bf Ram}'' part of $\fa$, i.e.\ $\fa'$ is the ideal of $A$ so that
\[
{\rm ord}_v(\fa') = 
\begin{cases}
{\rm ord_v}(\fa), & \text{ if $v \notin {\bf Ram}\cup \{\infty\}$;} \\
0, & \text{ otherwise.}
\end{cases}
\]
the following lemma holds.

\begin{lemma}\label{lem: E-decomp}
Taking $\star \in \{\times,\circ\}$,
we may decompose $\cE_\fa^\star$ into
\begin{equation}\label{eqn: Ea-decomp}
\cE_\fa^\star = \coprod_{\subfrac{\bar{g} \in \cK(\fn)\backslash (\cD^\infty \cap D^*(\bA_f))/\cK(\fn)}{\text{\rm Nr}(g)\cdot \cO^{\infty}\cap k = \fa}}\ \ 
\coprod_{\bar{\kappa} \in \cK(\fn,g)\backslash \cK(\fn)} \cE^\star(\fn,g\kappa).
\end{equation}
Consequently, we have the bijection:
\[
\coprod_{\subfrac{\bar{g} \in \cK(\fn)\backslash (\cD^\infty \cap D^*(\bA_f))/\cK(\fn)}{\text{\rm Nr}(g)\cdot \cO^{\infty}\cap k = \fa}}\ \ 
\coprod_{\bar{\kappa} \in \cK(\fn,g)\backslash \cK(\fn)} D^*\backslash \cE^\star(\fn,g\kappa)/\cK(\fn\fa') \stackrel{\sim}{\longrightarrow} D^*\backslash \cE_\fa^\star /\cK(\fn\fa').
\]
\end{lemma}

\begin{proof}
Given $(\gamma,b) \in \cE_\fa^\star$, 
we may take $g = b^{-1}\gamma b \in \cD^\infty \cap D^*(\bA_f)$ and get $(\gamma,b) \in \cE^\star (\fn,g)$.
Thus it remains to show that the right hand side of the equality \eqref{eqn: Ea-decomp} is indeed a disjoint union.

Given $g,g' \in \cD^\infty \cap D^*(\bA_f)$ with $\text{Nr}(g)\cdot \cO^{\infty} \cap k = \text{Nr}(g')\cdot \cO^{\infty}\cap k = \fa$,
suppose $\cE^\star(\fn,g)\cap \cE^\star(\fn,g')$ is non-empty, i.e.\ there exists $(\gamma,b) \in \cE^\star(\fn,g)\cap \cE^\star(\fn,g')$.
Then $b^{-1}\gamma b \in \cK(\fn)g \cap \cK(\fn)g'$, which implies that $\cK(\fn)g = \cK(\fn)g'$.
Hence
\[
\cE_\fa^{\star} = \coprod_{\subfrac{\bar{g} \in \cK(\fn)\backslash \cD^\infty \cap D^*(\bA_f)}{\text{Nr}(g)\cdot \cO^{\infty}\cap k = \fa}} \cE^{\star}(\fn,g),
\]
and the desired decomposition in \eqref{eqn: Ea-decomp} follows from:
\[
\cK(\fn)g \cK(\fn) = \coprod_{\bar{\kappa} \in \cK(\fn,g)\backslash \cK(\fn)} \cK(\fn)g\kappa, \quad \forall g \in \cD^\infty \cap D^*(\bA_f).
\]
\end{proof}

Next, we look closely into $D^*\backslash \cE_\fa^\star /\cK(\fn\fa')$ for $\star = \times$ or $\circ$.
We may decompose $D^*\backslash \cE_\fa^{\times} / \cK(\fn\fa')$ into 
\[
D^*\backslash \cE_\fa^{\times} / \cK(\fn\fa') = \hspace{-0.3cm} \coprod_{[\gamma,b] \in D^*\backslash \cE_\fa^{\star} / \cK} 
\big\{D^*(\gamma,b\kappa)\cK(\fn\fa')\ \big|\ \bar{\kappa} \in (b^{-1}C_D(k(\gamma))b \cap \cK)\backslash \cK/\cK(\fn\fa')\big\},
\]
where $C_D(k(\gamma))$ is the centralizer of $k(\gamma)$ in $D$.
Applying Eichler's theory of optimal embeddings of imaginary $A$-orders into $D$ in Appendix~\ref{sec: Eich}, we know that $D^*\backslash \cE_\fa^{\times}/\cK$ is finite (see Remark~\ref{rem: finiteness}).
Hence 
\[
\#\big(D^*\backslash \cE_\fa^{\times}/\cK(\fn\fa')\big) = \sum_{[\gamma,b] \in D^*\backslash \cE_\fa^{\times}/\cK} \frac{[\cK:\cK(\fn\fa')]}{\#(b^{-1}k(\gamma)^*b \cap \cK)}.
\]

Finally, by Lemma~\ref{lem: E-decomp}, we obtain that
\begin{eqnarray*}
&& \#\big(D^*\backslash \cE_\fa^{\times}/\cK(\fn\fa')\big) \\
& = & \sum_{\bar{g}} \sum_{\bar{\kappa} \in K(\fn,g)\backslash K(\fn)} \#\big(D^*\backslash \cE^{\times}(\fn,g\kappa)/\cK(\fn\fa')\big) \\
& = & \sum_{\bar{g}} \sum_{\bar{\kappa} \in \cK(\fn,g)\backslash K(\fn)} [\cK(\fn,g\kappa):\cK(\fn\fa')] \cdot  \#\big(D^*\backslash \cE^{\times}(\fn,g\kappa)/\cK(\fn,g\kappa)\big) \\
& = & [\cK(\fn):\cK(\fn\fa')] \cdot \sum_{\bar{g}} \#\big(D^*\backslash \cE^{\times}(\fn,g)/\cK(\fn,g)\big) \quad (\text{by Remark~\ref{rem: ind-k}}),
\end{eqnarray*}
where $\bar{g}$ runs through the double cosets in $\cK(\fn)\backslash (\cD^\infty \cap D^*(\bA_f))/\cK(\fn)$ with $\text{Nr}(g) \cdot \cO^{\infty}\cap k = \fa$.
Therefore 
\begin{eqnarray}\label{eqn: tech}
&& \frac{1}{[\cK:\cK(\fn)]} \cdot \sum_{\bar{g}}\#\big(D^*\backslash \cE^{\times}(\fn,g)/\cK(\fn,g)\big) \nonumber \\
&=& \frac{1}{[\cK:\cK(\fn\fa')]}\cdot \#\big(D^*\backslash \cE_\fa^{\times} /\cK(\fn\fa')\big) \nonumber \\
&=& \frac{1}{q-1}\cdot \sum_{[\gamma,b] \in D^*\backslash \cE_\fa^{\times}/\cK}\frac{q-1}{\#(b^{-1}k(\gamma)^* b \cap \cK)} \nonumber \\
&=& \frac{1}{q-1}\cdot \sum_{[\gamma,b] \in D^*\backslash \cE_\fa^{\times}/\cK}\frac{q-1}{\#(k(\gamma)^* \cap b\cK b^{-1})}.
\end{eqnarray}

Recall that $\cE^\circ(\fn,g)$ is non-empty if and only if $g \in k^*\cdot \cK(\fn)$, which is equivalent to $X(\fn)= X(\fn,g)$ by Corollary~\ref{cor: sel-int}.
For $\star \in \{\times,\circ \}$ and $(\gamma,b) \in \cE^\star_\fa$, we put
\[
i(\gamma,b) :=
\begin{cases}
    \displaystyle 
    \frac{q-1}{\#(k(\gamma)^*\cap b\cK b^{-1})}, & \text{ if $\star =\times$}, 
    \vspace{5pt} \\
\displaystyle{\frac{q-1}{r}}\cdot \mu_\infty(\Gamma_b \backslash \GL_r(k_\infty)/k_\infty^*), & \text{ if $\star = \circ $}.
\end{cases}
\]
By Corollary~\ref{cor: emb}, the equation~\eqref{eqn: E-P-C}, and the projection formula in Example~\ref{ex: proj-formula}, we obtain the following.

\begin{proposition}\label{prop: int-formula}
Suppose $r$ is a prime number different from the characteristic of $k$.
Let $\fa$ be a nonzero ideal of $A$.
We have that
\[
i(\cZ\cdot \cZ_\fa) = \frac{r}{q-1} \cdot \sum_{[\gamma,b] \in D^*\backslash \cE_\fa / \cK} i(\gamma,b).
\]

\end{proposition}

\begin{proof}
By Theorem~\ref{ex: proj-formula}, we have that
\begin{eqnarray*}
i(\cZ\cdot \cZ_\fa) &=& \frac{1}{[\cK: \cK(\fn)]} \cdot \sum_{\subfrac{g \in \cK(\fn)\backslash \cD^\infty \cap D^*(\bA_f)/\cK(\fn)}{\text{Nr}(g)\cO^{\infty}\cap k = \fa}}i(\cZ(\fn) \cdot \cZ(\fn,g)) \\
&=& \frac{1}{[\cK: \cK(\fn)]} \cdot \sum_{\subfrac{g \in \cK(\fn)\backslash \cD^\infty \cap D^*(\bA_f)/\cK(\fn)}{g \notin k^*\cK(\fn),\ \text{Nr}(g)\cO^{\infty}\cap k = \fa}}i(\cZ(\fn) \cdot \cZ(\fn,g)) \\
&&\hspace{-0.27cm}+\frac{1}{[\cK: \cK(\fn)]} \cdot 
\sum_{\subfrac{\gamma \in A}{\gamma^r A = \fa}} \chi\big(X(\fn)\big).
\end{eqnarray*}
By Corollary~\ref{cor: emb} and Remark~\ref{rem: G-B-T}~(1),
the right-hand-side of the above equality can be written as
\begin{eqnarray*}
&& \frac{r}{[\cK: \cK(\fn)]} \cdot \sum_{\subfrac{g \in \cK(\fn)\backslash \cD^\infty \cap D^*(\bA_f)/\cK(\fn)}{g \notin k^*\cK(\fn),\ \text{Nr}(g)\cO^{\infty}\cap k = \fa}}
\#\big(D^*\backslash \cE^{\times}(\fn,g)/\cK(\fn,g)\big)\\
&+&  
\sum_{\subfrac{\gamma \in A}{\gamma^r A = \fa}} 
\sum_{b \in D^*\backslash D^*(\bA_f)/\cK}
\mu_\infty\big(\Gamma_b\backslash \text{GL}_r(k_\infty)/k_\infty^*\big).
\end{eqnarray*}
Note that $D^*\backslash \cE_\fa^\circ/\cK$ is in bijection with
$\{\gamma \in A \mid \gamma^r A = \fa\} \times D^*\backslash D^*(\bA_f)/\cK$.
 Therefore by the equation~\eqref{eqn: tech}, we obtain that
\begin{eqnarray*}
i(\cZ\cdot \cZ_\fa)
&=& \frac{r}{q-1}\cdot \sum_{(\gamma,b) \in D^*\backslash \cE_\fa^{\times}/\cK} i(\gamma,b) 
+  \frac{r}{q-1}\cdot  \sum_{(\gamma,b) \in D^*\backslash \cE_\fa^{\circ}/\cK} i(\gamma,b) \\
&=& \frac{r}{q-1}\cdot 
\sum_{(\gamma,b) \in D^*\backslash \cE_\fa/\cK}  i(\gamma,b).
\end{eqnarray*}
\end{proof}

In the next section, we shall express the right hand side of the above equality in terms of class numbers of imaginary fields, and derive an Kronecker-Hurwitz-type class number relation.

\section{Class number relation}\label{sec: CR}

Keep the notation as in the last section.
Suppose $\cE_\fa$ is non-empty, which implies that
$\fa$ is a principal ideal of $A$.
In this case, for each $(\gamma,b) \in \cE_\fa$,
the condition $b^{-1}\gamma b \in \cD^\infty$ is equivalent to
\[
\gamma \in b\cD^\infty b^{-1} \cap D =:O_{D,b},
\]
where $O_{D,b}$ is a maximal $A$-order of $D$.
This implies that $\gamma$ is integral over $A$.
Suppose $(\gamma,b) \in \cE_\fa^\times$.
Let 
\[
f_\gamma(x) = x^r-c_1 x^{r-1}+\cdots + (-1)^{r-1}c_{r-1}x+(-1)^r c_r \in A[x]
\]
be the minimal polynomial of $\gamma$ over $k$.
Since $k(\gamma)$ is imaginary over $k$ of degree $r$, the polynomial $f_\gamma(x)$ remains irreducible over $k_\infty$.
Hence every root of $f_\gamma(x)$ in $\bC_\infty$ has the same absolute value.
Notice that $c_r \cdot A = \text{Nr}(\gamma)\cdot A = \fa$, which gives us that $|c_i|_\infty \leq q^{\deg \fa}$ for $i=1,...,r$.
Therefore the possibility of the minimal polynomial $f_\gamma \in A[x]$ is finite for $(\gamma,b)$ running through all pairs in $\cE_\fa^\times$.

On the other hand, given $\vec{c}:=(c_1,...,c_r) \in A^r$,
let 
\[
f_{\vec{c}}(x):= x^r-c_1x^{r-1}+\cdots + (-1)^{r-1}c_{r-1}x+(-1)^r c_r \quad \in A[x].
\]
Suppose that:
\begin{itemize}
\item[(i)] $c_r\cdot A = \fa$;
\item[(ii)]
$K_{\vec{c}}:= k[x]/(f_{\vec{c}}(x))$
is an imaginary field extension over $k$;
\item[(iii)]
there exists a $k$-algebra embedding $\phi: K_{\vec{c}}\hookrightarrow D$.
\end{itemize}
Put
$R_{\vec{c}}:= A[x]/(f_{\vec{c}}(x))$,
which is an $A$-order in $K_{\vec{c}}$, and $\bar{x}$ is the coset in $R_{\vec{c}}$ represented by $x$.
For each $b \in D^*(\bA_f)$ satisfying that
\[
R_{\vec{c}} \subset \phi^{-1}(O_{D,b}),
\]
set $\gamma_\phi := \phi(\bar{x}) \in O_{D,b}$.
Then $(\gamma_\phi,b)$ lies in $\cE_\fa^\times$.

Let
\[
\cE(\vec{c}):=
\{b \in D^*(\bA_f) \mid R_{\vec{c}}\subset \phi^{-1}(O_{D,b})\},
\]
which is invariant by left multiplication of $\phi(K_{\vec{c}}^*)$ and right multiplication of $\cK$.
For each $(\gamma,b) \in \cE_\fa$ with $f_\gamma = f_{\vec{c}}$,
by the Noether-Skolem theorem, there exists $\gamma_0 \in D^*$, which is unique up to left mulltiplication of $\phi(K_{\vec{c}}^*)$, such that $\gamma_\phi = \gamma_0 \gamma \gamma_0^{-1}$.
Then the condition $b^{-1}\gamma b \in \cD^\infty$
is equivalent to $\gamma_\phi \in O_{D,\gamma_0 b}$.
Hence we have a well-defined bijective map
\begin{equation}\label{eqn: fin}
D^*\backslash \cE_\fa^\times /\cK \longrightarrow \coprod_{\subfrac{\vec{c} \in A^r:\  c_r A = \fa,}{K_{\vec{c}}/k \text{ is imaginary}}} \phi(K_{\vec{c}}^*)\backslash \cE(\vec{c}) /\cK 
\end{equation}
which sends $[\gamma,b]$ to $[b]_{\vec{c}} \in \phi(K_{\vec{c}}^*)\backslash \cE(\vec{c}) /\cK$
when $f_\gamma = f_{\vec{c}}$.

Now, for an $A$-order $R$ of $K_{\vec{c}}$ with $R_{\vec{c}}\subset R$,
put
\[
\cE^o_\phi(R):=\{b \in D^*(\bA_f)\mid \phi^{-1}(O_{D,b}) = R\}.
\]
Then
\begin{equation}\label{eqn: fin-3}
\phi(K_{\vec{c}}^*)\backslash \cE(\vec{c})/\cK = \coprod_{R: R_{\vec{c}}\subset R} 
\phi(K_{\vec{c}}^*)\backslash \cE^o_\phi(R)/\cK.
\end{equation}
Note that for $b \in \cE^o_\phi(R)$, one has that
\[
k(\gamma_{\phi})^* \cap b\cK b^{-1}
= \phi(K_{\vec{c}})^* \cap O_{D,b}^* = \phi(R^*).
\]
Therefore by the equation~\eqref{eqn: tech} and Proposition~\ref{prop: OE-HC},
we obtain that 
\begin{eqnarray*}
\sum_{[\gamma,b] \in D^*\backslash \cE_\fa^\times/\cK} \frac{\#(\bF_q^\times)}{\#(k(\gamma)^*\cap b\cK b^{-1})}
&=& \sum_{\subfrac{\vec{c} \in A^r:\  c_r A = \fa,}{K_{\vec{c}} \text{ is imag.}}} 
\sum_{R: R_{\vec{c}}\subset R}
\frac{\#\big(\phi(K_{\vec{c}}^*)\backslash \cE^o_\phi(R)/\cK\big)}{\#(R^*/\bF_q^*)} \\
&=& \sum_{\subfrac{\vec{c} \in A^r:\  c_r A = \fa,}{K_{\vec{c}} \text{ is imag}}} H^D(\vec{c}),
\end{eqnarray*}
where $H^D(\vec{c})$ is the \textit{modified Hurwitz class number} with respect to $D$ introduced in Definition~\ref{def: MHC}.

Finally,
set
\begin{equation}\label{eqn: H(c)}
H^D(0):=
\#\text{Pic}(A)\cdot 
\prod_{i=1}^{r-1}\zeta_k(-i) \cdot \prod_{v \in {\bf Ram}\cup\{\infty\}} \, \prod_{i=1}^{r-1}(1-q_v^i).
\end{equation}
As $D^*\backslash \cE_\fa^\circ/\cK$ is in bijection with $\{c \in A\mid c^r \cdot A = \fa\} \times D^*\backslash D^*(\bA_f)/\cK$, which is finite,
by Proposition~\ref{prop: self-int} and Remark~\ref{rem: G-B-T}~(1) one gets
\[
\sum_{(\gamma,b) \in D^*\backslash \cE_\fa^\circ/\cK} i(\gamma,b) = \sum_{\subfrac{c \in A}{c^r A = \fa}} H^D(0).
\]
Therefore Proposition~\ref{prop: int-formula} leads to the following class number relation:

\begin{theorem}\label{thm: C-R}
Let $r$ be a prime number distinct from the characteristic of $k$. 
For each nonzero ideal $\fa$ of $A$, we have
\[
i(\cZ\cdot \cZ_\fa) = \frac{r}{q-1}\cdot \left(\sum_{\subfrac{\vec{c} \in A^r:\  c_r A = \fa,}{K_{\vec{c}} \text{\rm\ is imag.}}} H^D(\vec{c})
+\sum_{\subfrac{c \in A}{c^r \cdot A = \fa}}H^D(0)\right).
\]    
\end{theorem}

\begin{remark}
Suppose $r=2$ and the characteristic of $k$ is odd.
Given $a \in A$, we may denote by $a\prec 0$ if $k(\sqrt{a})$ is an imaginary quadratic field extension over $k$.
For such $a$, we let
\[
H^D(a):= H^D((0,-a)).
\]
Then Theorem~\ref{thm: C-R} can be reformulate as the following:
for each nonzero ideal $\fa$ of $A$, we have
\[
i(\cZ\cdot \cZ_{\fa}) = \frac{2}{q-1}\cdot\sum_{\subfrac{t,a \in A}{aA = \fa,\ t^2-4a\preceq 0}} H^D(t^2-4a).
\]
Furthermore, set
\[
\text{vol}(\cZ) := 2\cdot \#{\rm Pic}(A)\cdot (q_\infty-1) \cdot \zeta_k(-1) \cdot \prod_{v \in {\bf Ram}}(1-q_v).
\]
By employing the ``Weil representation'' as illustrated in \cite[Sec.~3]{GWei0} (or alternatively, by combining the work of \cite{CLWY} with the relations among Hurwitz class numbers established in \cite[Sec.~2]{GWei}), we are able to construct a ``harmonic'' theta series on $\GL_2(\bA)$ whose nonzero (resp.\ constant) Fourier coefficients come from the intersection numbers $i(\cZ\cdot \cZ_\fa)$ for nonzero ideal $\fa$ of $A$ (resp.\ $\text{vol}(\cZ)$).
Although the approach in \cite{GWei0} seems not applicable for $r>2$, we believe that, after further work, this phenomenon remains valid.
\end{remark}


\appendix

\section{Eichler's theory of optimal embeddings}\label{sec: Eich}

Recall that $D$ is a division algebra  over $k$ with $[D:k]=r^2$.
Here we review Eichler's theory of optimal embeddings from an $A$-order $R$ of an imaginary extension $K$ over $k$ into the division algebra $D$.
For simpliciy, we always assume that $r$ is a prime number in this section, which is sufficient for our purpose.

\subsection{Local theory}

Fix a place $v$ of $k$ with $v \neq \infty$,
let
\[
K_v := K \otimes_k k_v,\quad
D_v := D \otimes_k k_v, \quad
O_{D,v} := O_D\otimes_A \cO_v.
\]
As $r$ is a prime number, either $D_v$ is division or $D_v \cong \bM_r(k_v)$.
Thus there always exists an embedding from $K_v$ into $D_v$ when either $D_v\cong \bM_r(k_v)$ or $K_v$ is a field (see \cite[Theorem~1.1~(1)]{SYY}).
Fix an embedding $\phi_v: K_v \hookrightarrow D_v$.
For each $\cO_v$-order $R_v$ in $K_v$, let 
\[
\cE_{\phi_v}^o(R_v):= \{ b_v \in D_v^* \mid \phi_v^{-1}(bO_{D,v}b^{-1}) = R_v\}.
\]
We shall determine the finiteness of the double coset space
\[
\phi_v(K_v^*)\backslash \cE_{\phi_v}^o(R_v)/\cK_v,
\]
where $\cK_v = O_{D,v}^*$, the $v$-component of $\cK$.

\begin{lemma}\label{lem: mu-D-1}
Suppose $D_v$ is division and $K_v$ is a field. 
Then
\begin{align*}
\#\big(\phi_v(K_v^*)\backslash \cE_{\phi_v}^o(R_v)/\cK_v\big)  \\
& \hspace{-2cm} = \mu_v^D(R_v):= 
\begin{cases}
r, & \text{ if $v$ is unramified in $K_v$ and $R_v$ is maximal in $K_v$,}\\
1, & \text{ if $v$ is ramified in $K_v$ and $R_v$ is maximal in $K_v$,}\\
0, & \text{ if $R_v$ is not maximal in $K_v$.}
\end{cases}
\end{align*}
\end{lemma}

\begin{proof}
When $D_v$ is division and $K_v$ is a field, one knows that $O_{D,v}$ is the unique maximal $\cO_v$-order in $D_v$, for which $b_v O_{D,v}b_v^{-1} = O_{D,v}$.
Moreover, every element in $D_v$ which is integral over $\cO_v$ is contained in $O_{D,v}$.
Therefore for every $b_v \in D_v^*$,
\[
\phi_v^{-1}(b_v O_{D,v} b_v^{-1})
= \phi_v^{-1}(O_{D,v}) = R_{K,v},
\quad \text{the maximal $\cO_v$-order in $K_v$}.
\]
Thus 
$\#\big(\phi_v(K_v^*)\backslash \cE_{\phi_v}^o(R_v)/\cK_v\big) = 0$
if $R_v$ is not maximal in $K_v$.

When $R_v = R_{K,v}$, we get that
$\cE_{\phi_v}^o(R_v) = D_v^*$.
Hence
\begin{align*}
\#\big(\phi_v(K_v^*)\backslash \cE_{\phi_v}^o(R_v)/\cK_v\big)
&= \#\big(\phi_v(K_v^*)\backslash D_v^*/\cK_v\big) \\
& = 
\begin{cases}
    r, & \text{if $v$ is unramified in $K_v$}\\
    1, & \text{if $v$ is (totally) ramified in $K_v$.}
\end{cases}
\end{align*}
This completes the proof.
\end{proof}

\begin{lemma}\label{lem:local-multi}
Suppose $D_v\cong \bM_r(k_v)$.
Then the cardinality of $\phi_v(K_v^*)\backslash \cE_{\phi_v}^o(R_v)/\cK_v$ is finite and positive.
In particular, if $R_v$ is maximal in $K_v$, then
\[
\#\big(\phi_v(K_v^*)\backslash \cE_{\phi_v}^o(R_v)/\cK_v\big) = 1.
\]
\end{lemma}

\begin{proof}
Under the isomorphism $D_v\cong \bM_r(k_v)$, we may identify $O_{D,v}$ with $\text{Mat}_r(\cO_v)$ without loss of generality.
Let $V = k_v^r$, which is equipped with a left module structure over $\bM_r(k_v) = \text{End}_{k_v}(V)$.
An $R_v$-lattice $L$ is a free $\cO_v$-submodule of $V$ which is invariant under the multiplication by $\phi_v(R_v)$, and $L$ is {\it optimal} if $\phi_v(R_v) = \phi_v(K_v) \cap \text{End}_{\cO_v}(L)$.
Two $R_v$-lattices $L_1$ and $L_2$ are isomorphic if there exists $b \in \text{Aut}_{k_v}(V)$ so that $b L_1 = L_2$, and 
\[
b\big(\phi_v(\alpha)\lambda\big) = \phi_v(\alpha)\cdot b(\lambda), \quad \forall \alpha \in K_v,\ \lambda \in L_1.
\]
Then $b \in \phi_v(K_v^*)$, and we have a bijection
\[
\phi_v(K_v^*)\backslash \cE_{\phi_v}^o(R_v)/\cK_v \stackrel{\sim}{\longrightarrow} \{\text{isomorphism classes of optimal $\cO_v$-lattices in $V$}\}
\]
which is induced by sending $b$ to the optimal $R_v$-lattice $b\cdot \cO_v^r$ for every $b \in \cE_{\phi_v}^o(R_v)$.
Thus it suffices to count the isomorphism classes of optimal $R_v$-lattices in $V$.

Notice that when $R_v$ is maximal in $K_v$, every $R_v$-lattice in $L$ must be free of rank one over $R_v$.
Hence there is only one isomorhism class.
In general, for each optimal $R_v$-lattice $L$ in $K_v$, consider $\tilde{L}:= \phi_v(R_{K,v})\cdot L \subset V$, which is an $R_{K,v}$-lattice.
On the other hand, there exists $c_v \in \bZ_{\geq 0}$ such that $\fp_v^{c_v} \cdot R_{K,v} \subset R_v$ (the smallest choice of `$c_v$' refers to the ``conductor'' of $R_v$), whence $\fp_v^{c_v} \cdot \tilde{L} \subset L$.
Since $\tilde{L}/\fp_v^{c_v} \tilde{L}$ has finite cardinality, we obtain that the isomorphism classes of $R_v$-lattices in $V$ is finite (and bounded by $\#(R_{K,v}/\fp_v^{c_v} R_{K,v})$).
\end{proof}

\begin{remark}\label{rem: local-mu}
Suppose $D_v \cong \bM_r(k_v)$. Let $\mu_v^D(R_v)$ be the cardinality of the isomorphism classes of optimal $R_v$-lattices in $V$.
Then the above proof shows that $\mu_v^D(R_v)$ is finite positive, and $\mu_v^D(R_v)=1$ for almost all $v$.
In particular,
\[
\#\big(\phi_v(K_v^*)\backslash \cE_{\phi_v}^o(R_v)/\cK_v\big) = \mu_v^D(R_v).
\]
\end{remark}

\subsection{Global theory and modified Hurwitz class numbers}

First, the condition for the existence of a $k$-algebra embedding from $K$ into $D$ is determined by the following.

\begin{theorem}
(See \cite[Proposition~A.1]{PR}) Let $D$ be a central simple algebra of degree $r$ over $k$ and $K$ is a field extension of degree $r$ over $k$.
Then the local-global principle holds, i.e.\ $K$ can be embedded into $D$ if and only if $K_v$ can be embedded into $D_v$ for every place $v$ of $K$.
\end{theorem}

Now, let $D$ be a division algebra over $k$ and $r$ is a prime number. 
Suppose an embedding $\phi: K \hookrightarrow D$ exists, which corresponds to embeddings $\phi_v:K_v\hookrightarrow D_v$ for all finite places $v$ of $k$.
Let $R$ be an $A$-order of $K$.
Put
\[
\cE^o_{\phi}(R):=\{b \in D^*(\bA_f)\mid \phi^{-1}(O_{D,b}) = R\}.
\]
Here $O_{D,b}$ is defined in the beginning of Section~\ref{sec: CR}.
We shall express the cardinality of the double coset space $\phi(K^*)\backslash \cE_\phi^o(R)/\cK$ as ``modified class number'' of $R$.

For each finite place $v$ of $k$, let $R_v:=R\otimes_A \cO_v$.
For $b = (b_v)_{v\neq \infty} \in D^*(\bA_f)$, one has that 
$b \in \cE_\phi^o(R)$
if and only if $b_v \in \cE_{\phi_v}^o(R_v)$ for every $v \neq \infty$.
Define $K^*(\bA_f):= (K\otimes_k \bA_f)^*$.
Then Lemma~\ref{lem:local-multi} ensures that 
\[
\phi(K^*(\bA_f))\backslash \cE^o_{\phi}(R)/\cK
= \prod_v \phi_v(K_v^*)\backslash \cE^o_{\phi,v}(R_v)/\cK_v.
\]
On the other hand, 
let $h(R)$ be the class number of $R$, i.e.\
\[
h(R):= \#(K^*\backslash K^*(\bA_f)/\widehat{R}^*),\quad \text{ where $\widehat{R}:= \prod_{v\neq \infty} R_v$}.
\]
We have the following.

\begin{lemma}\label{lem: fiber-card}
Under the canonical surjective map 
\[
\phi(K^*)\backslash \cE_\phi^o(R)/\cK\longrightarrow \phi(K^*(\bA_f))\backslash \cE_\phi^o(R)/\cK,
\]
the cardinality of each fiber is equal to $h(R)$.
\end{lemma}

\begin{proof}
Given $b \in \cE_\phi^o(R)$, the fiber of $b$ can be identified with the double coset space
\[
\phi(K^*)\Big\backslash \phi(K^*(\bA_f))\Big/\big(\phi(K^*(\bA_f))\cap b\cK b^{-1}\big),
\]
and the condition of $b$ lying in $\cE_\phi^o(R)$ implies that
\[
\phi(K^*(\bA_f))\cap b\cK b^{-1} = \phi(\widehat{R}^*).
\]
Therefore the result follows.
\end{proof}

Consequently, let $h^D(R)$ be the following \textit{modified class number} (with respect to $D$):
\[
h^D(R) := h(R)\cdot \prod_{v \neq \infty}\mu_v^D(R_v),
\]
where $\mu_v^D(R_v)$ is the introduced in Lemma~\ref{lem: mu-D-1} when $D_v$ is division and  Remark~\ref{rem: local-mu} when $D_v \cong \bM_r(k_v)$.
We then conclude that:

\begin{proposition}\label{prop: finiteness}
The double coset space $\phi(K^*)\backslash \cE_\phi^o(R)/\cK$ is a finite set, and its cardinality is equal to $h^D(R)$.
\end{proposition}

\begin{proof}
By Lemma~\ref{lem:local-multi} and Remark~\ref{rem: local-mu}, we get
\[
\#\big(\phi(K^*(\bA_f))\backslash \cE_\phi^o(R)/\cK\big)
= \prod_v \big(\phi(K_v^*)\backslash \cE^o_{\phi,v}(R_v)/\cK_v\big) = \prod_v \mu_v^D(R_v) \quad < \infty.
\]
Therefore Lemma~\ref{lem: fiber-card} ensures the finiteness of
$\phi(K^*)\backslash \cE_\phi^o(R)/\cK$, and 
\[
\#\big(\phi(K^*)\backslash \cE_\phi^o(R)/\cK\big) = h(R) \cdot \#\big(\phi(K^*(\bA_f))\backslash \cE_\phi^o(R)/\cK\big) = h^D(R) 
\]
as desired.
\end{proof}

\begin{remark}\label{rem: finiteness}
Given a non-zero ideal $\fa$ of $A$ so that $\cZ$ is not a component of $\cZ_\fa$,
Then Proposition~\ref{prop: finiteness}, the equations~\eqref{eqn: fin} and \eqref{eqn: fin-3} guarantee the finiteness of $D^*\backslash \cE_\fa^{\times} /\cK$.
\end{remark}

\begin{definition}\label{def: MHC}
Given $\vec{c} \in A^r$ satisfying that $K_{\vec{c}}$ is imaginary,
the \textit{modified Hurwitz class number of $R_{\vec{c}}$} (with respect to $D$) is given by
\[
H^D(\vec{c}):= \sum_{R: R_{\vec{c}}\subset R}\frac{h^D(R)}{\#(R^*/\bF_q^*)}.
\]
\end{definition}

We finally arrive at:

\begin{proposition}\label{prop: OE-HC}
\[
\sum_{R: R_{\vec{c}}\subset R}
\frac{\#\big(\phi(K_{\vec{c}}^*)\backslash \cE^o_\phi(R)/\cK\big)}{\#(R^*/\bF_q^*)} = H^D(\vec{c}).
\]
\end{proposition}

\printbibliography

@article{B,
AUTHOR = {Briney, R. E.},
title={Intersection theory on quotients of algebraic varieties},
JOURNAL = {Amer. J. of Math},
    VOLUME = {84},
      YEAR = {1962},
       PAGES = {217-238}
}

@Inbook{BS,
  AUTHOR= {Blum A., Stuhler U.},
  TITLE={Drinfeld modules and elliptic sheaves},
  BOOKTITLE={Vector Bundles on Curves --- New Directions: Lectures given at the 3rd Session of the Centro Internazionale Matematico Estivo (C.I.M.E.) held in Cetraro (Cosenza), Italy, June 19--27, 1995},
  YEAR={1997},
  PUBLISHER={Springer Berlin Heidelberg},
  PAGES={110--188},
  ISBN={978-3-540-49701-1},
 % DOI={10.1007/BFb0094426}
}

@article{Chev,
     author = {Chevalley, Claude},
     title = {Les classes d'\'equivalence rationnelle, {I}},
     journal = {S\'eminaire Claude Chevalley},
     note = {talk:2},
     publisher = {Secr\'etariat math\'ematique},
     volume = {3},
     year = {1958},
     language = {fr}
}

@article{DH,
AUTHOR = {Deligne P., Husem\"oller D.},
title={Survey of Drinfeld modules},
JOURNAL = {Contemporary Math.},
    VOLUME = {67},
      YEAR = {1987},
     PAGES = {25–91}
}

@article{DM,
AUTHOR = {Deligne P., Mumford D.},
title={The Irreducibility of the Space of Curves of Given Genus},
JOURNAL = {Publ. Math. IHES},
  FJOURNAL = {Publications Mathèmatiques de l’Institut des Hautes Scientifiques},
    VOLUME = {36},
      YEAR = {1969},
     PAGES = {75–109},
     %  DOI = {10.1007/s00222-016-0708-y},
}

@article{Dr1,
    author = {Drinfeld, V.G.} ,
    title = {Elliptic modules},
    journal = {Math. USSR Sbornik23, 561–592},
    year = {1974},
}

@book{GW,
AUTHOR= { G\"ortz U., Wedhorn T.}, 
TITLE= {Algebraic Geometry},
SERIES= {Advanced Lectures in Mathematics},
YEAR={2010},
PUBLISHER= {Vieweg+Teubner Verlag Wiesbaden},
PAGES={IV, 615}, 
ISBN={978-3-8348-9722-0},

}

@book{H,
AUTHOR= { Hartshorne, R.}, 
TITLE= {Algebraic Geometry},
SERIES= {Graduate Texts in Mathematics},
YEAR={1977},
PUBLISHER= {Springer New York},
PAGES={XVI, 496}, 
ISBN={978-0-387-90244-9},

}

@article {LRS,
    AUTHOR = {Laumon G., Rapoport M.,Stuhler U.},
     TITLE = {$\mathcal{D}$-elliptic sheaves and the Langlands correspondence},
   JOURNAL = {Invent. Math.},
  FJOURNAL = {Inventiones Mathematicae},
    VOLUME = {113},
      YEAR = {1993},
     PAGES = {217-338}
}

@article {Pap,
    AUTHOR = {Papikian, M.},
     TITLE = {Drinfeld–Stuhler modules},
   JOURNAL = {Res Math Sci},
   FJOURNAL={Research in the Mathematical Sciences}, 
   VOLUME = {5},
    YEAR = {2018}
}

@article {SS,
    AUTHOR = {Schneider P., Stuhler U.},
     TITLE = {The cohomology of p-adic symmetric spaces},
   JOURNAL = {Invent. Math.},
   FJOURNAL={Inventiones Mathematicae}, 
   VOLUME = {105},
    YEAR = {1991},
PAGES = {47-122}
}

@book {Serre,
    AUTHOR = {Serre, J.-P.}, 
     TITLE = {Groupes alg\'ebriques et corps de classes},
     SERIES = {Actualit\'es Scientifiques et Industrielles},
     PUBLISHER = {Hermann (Paris)},
     NUMBER = {1264},
      YEAR = {1959},
      ISBN = {9782705612641}
}

@book {Ul,
    AUTHOR = {\"Ulkem, \"O.}, 
     TITLE = {Uniformization of generalized $\cD$-elliptic sheaves, PhD thesis},
     PUBLISHER = {Heidelberg University}, 
      YEAR = {2020},
      URL = {http://www.ub.uni-heidelberg.de/archiv/29531},
}

@article {SYY,
    AUTHOR = {Shih, S.-C. and Yang, T.-C. and Yu, C.-F.},
     TITLE = {Embeddings of fields into simple algebras over global fields},
   JOURNAL = {Asian J.~Math},
   FJOURNAL={The Asian Journal of Mathematics}, 
   VOLUME = {18},
   NUMBER = {2},
   YEAR = {2014},
   Pages = {365--386},
%     URL = {https://doi.org/10.4310/AJM.2014.v18.n2.a9},
}

@book{Ful,
AUTHOR= { Fulton, W.}, 
TITLE= {Intersection Theory, Second edition},
%SERIES= {Graduate Texts in Mathematics},
YEAR={1998},
PUBLISHER= {Springer New York},
PAGES={VIII, 470}, 
ISBN={978-0-387-98549-7},

}

@article {GWei0,
    AUTHOR = {Guo, J.-W. and Wei, F.-T.},
     TITLE = {On class number relations and intersections over function fields},
   JOURNAL = {Doc.~math.},
   FJOURNAL={Documenta Mathematica}, 
   VOLUME = {27},
   NUMBER = {},
     YEAR = {2022},
    PAGES = {1321--1368},
}

@article {GWei,
    AUTHOR = {Guo, J.-W. and Wei, F.-T.},
     TITLE = {On CM masses over global function fields},
   JOURNAL = {Math.~Z.},
   FJOURNAL={Mathematische Zeitschrift}, 
   VOLUME = {303},
   NUMBER = {29},
     YEAR = {2023},
%    PAGES = {},
}

@article {CLWY,
    AUTHOR = {Chuang, C.-Y. and Lee, T.-F. and Wei, F.-T. and Yu, J.},
     TITLE = {Brandt matrices and theta series over global function fields},
   JOURNAL = {Mem.~Amer.~Math.~Soc.},
   FJOURNAL={Memoirs of the American Mathematical Society}, 
   VOLUME = {237},
   NUMBER = {1117},
     YEAR = {2015},
    %PAGES = {},
}

@article {Hur,
    AUTHOR = {Hurwitz, A.},
     TITLE = {\"Uber Relationen zwischen Klassenzahlen bin\"arer quadratischer Formen von negativer Determinante},
   JOURNAL = {Math.~Ann.},
   FJOURNAL={Mathematische Annalen}, 
   VOLUME = {25},
   NUMBER = {},
    YEAR = {1885},
    PAGES = {157--196},
}

@article {H-W,
    AUTHOR = {Hirzebruch, F.~and Zagier, D.},
     TITLE = {Intersection numbers of curves on Hilbert modular surfaces
and modular forms of nebentypus},
   JOURNAL = {Invent.~math.},
   FJOURNAL={Inventiones Mathematicae}, 
   VOLUME = {36},
   NUMBER = {},
    YEAR = {1976},
    PAGES = {57--113},
}

@article {PR,
    AUTHOR = {Prasad, G. and Rapinchuk, A.~S.},
     TITLE = {Computation of the metaplectic kernel},
   JOURNAL = {Publ. Math. IHES},
   FJOURNAL={Publications mathématiques de l’I.H.É.S.}, 
   VOLUME = {84},
%   NUMBER = {29},
     YEAR = {1996},
    PAGES = {91--187},
}

@article {Kron,
    AUTHOR = {Kronecker, L.},
     TITLE = {\"Uber die Anzahl der verschiedenen Klassen quadratischer Formen von negativer Determinante},
   JOURNAL = {J.~Reine Angew.~Math.},
   FJOURNAL={Journal f\"ur die reine und angewandte Mathematik}, 
   VOLUME = {57},
   NUMBER = {},
    YEAR = {1860},
    PAGES = {248--255},
}

@article {Cog,
    AUTHOR = {Cogdell, J.~W.},
     TITLE = {Arithmetic cycles on Picard modular surfaces and modular forms of nebentypus},
   JOURNAL = {J.~Reine Angew.~Math.},
   FJOURNAL={Journal f\"ur die reine und angewandte Mathematik}, 
   VOLUME = {357},
   NUMBER = {},
     YEAR = {1985},
    PAGES = {115--137},
}

@article {Fun,
    AUTHOR = {Funke, J.},
     TITLE = {Heegner divisors and nonholomorphic modular forms},
   JOURNAL = {Compositio math.},
   FJOURNAL={Compositio Mathematica}, 
   VOLUME = {133},
   NUMBER = {},
     YEAR = {2002},
    PAGES = {289--321},
}

@article {Kud1,
    AUTHOR = {Kudla, S.\ S.},
     TITLE = {Intersection numbers for quotients of the complex $2$-ball and Hilbert modular forms},
   JOURNAL = {Invent.~math.},
   FJOURNAL={Inventiones mathematicae}, 
   VOLUME = {47},
   NUMBER = {2},
     YEAR = {1978},
    PAGES = {189--208},
}

@article {Kud2,
    AUTHOR = {Kudla, S.\ S.},
     TITLE = {Algebraic cycles on Shimura varieties of orthogonal type},
   JOURNAL = {Duke math.},
   FJOURNAL={Duke Mathematical Journal}, 
   VOLUME = {86},
   NUMBER = {1},
     YEAR = {1997},
    PAGES = {39--78},
}

@article {K-MI,
    AUTHOR = {Kudla, S.\ S.~and Millson, J.\ J.},
     TITLE = {The theta correspondence and harmonic forms.\ I},
   JOURNAL = {Math.~Ann.},
   FJOURNAL={Mathematische Annalen}, 
   VOLUME = {274},
   NUMBER = {},
     YEAR = {1986},
    PAGES = {353--378},
}

@article {K-MII,
    AUTHOR = {Kudla, S.\ S.~and Millson, J.\ J.},
     TITLE = {The theta correspondence and harmonic forms.\ II},
   JOURNAL = {Math.~Ann.},
   FJOURNAL={Mathematische Annalen}, 
   VOLUME = {277},
   NUMBER = {},
     YEAR = {1987},
    PAGES = {267--317},
}

@article {K-M,
    AUTHOR = {Kudla, S.\ S.~and Millson, J.\ J.},
     TITLE = {Intersection numbers of cycles on locally symmetric spaces and Fourier coefficients of holomorphic modular forms in several complex variables},
   JOURNAL = {Publications math\'ematiques de l'I.H.\'E.S.},
   FJOURNAL={}, 
   VOLUME = {71},
   NUMBER = {},
     YEAR = {1990},
    PAGES = {121--172},
}

@article {Ku,
    AUTHOR = {Kurihara, A.},
     TITLE = {Construction of p-adic unit balls and the Hirtzbruch proportionality},
   JOURNAL = {Amer.~J.~math},
   FJOURNAL={American Journal of Mathematics}, 
   VOLUME = {102},
   NUMBER = {3},
     YEAR = {1980},
    PAGES = {565--648},
}

@book {Weil,
     AUTHOR = {Weil, A.},
     TITLE = {Adeles and algebraic groups (with appendices by M. Demazure and T. Ono)},
     SERIES={Progress in Mathematics},
     VOLUME={23},
     YEAR={1982},
     PUBLISHER= {Birkhäuser, Boston, Mass.},
}

@article {WgY,
    AUTHOR = {Wang, T.-Y.~and Yu, J.},
     TITLE = {On class number relations over function fields},
   JOURNAL = {J.~Number Theory},
   FJOURNAL={Journal of Number Theory}, 
   VOLUME = {69},
   NUMBER = {},
     YEAR = {1998},
    PAGES = {181--196},
}

@article {W-Y,
    AUTHOR = {F.-T. Wei and C.-F. Yu},
     TITLE = {Mass formula of division algebras over global function fields},
   JOURNAL = {J.~Number Theory},
   FJOURNAL={Journal of Number Theory}, 
   VOLUME = {132},
   NUMBER = {3},
     YEAR = {2012},
    PAGES = {1170--1184},
}

@article {JKY,
    AUTHOR = {Yu, J.-K.},
     TITLE = {A class number relation over function fields},
   JOURNAL = {J.~Number Theory},
   FJOURNAL={Journal of Number Theory}, 
   VOLUME = {54},
   NUMBER = {},
     YEAR = {1995},
    PAGES = {318--340},
}

@article {Bae,
    AUTHOR = {Bae, S.},
     TITLE = {On the modular equation for Drinfeld modules of rank 2},
   JOURNAL = {J.~Number Theory},
   FJOURNAL={Journal of Number Theory}, 
   VOLUME = {42},
   NUMBER = {},
     YEAR = {1992},
    PAGES = {123--133},
}

@article {B-L,
    AUTHOR = {Bae, S.~and Lee, S.},
     TITLE = {On the coefficients of the Drinfeld modular equation},
   JOURNAL = {J.~Number Theory},
   FJOURNAL={Journal of Number Theory}, 
   VOLUME = {66},
   NUMBER = {},
     YEAR = {1997},
    PAGES = {85--101},
}

@article {Hsia,
    AUTHOR = {Hsia, L.-C.},
     TITLE = {On the coefficients of modular polynomials for Drinfeld modules},
   JOURNAL = {J.~Number Theory},
   FJOURNAL={Journal of Number Theory}, 
   VOLUME = {72},
   NUMBER = {},
     YEAR = {1998},
    PAGES = {236--256},
}

@article {FYZ1,
AUTHOR = {Feng, T. and Yun, Z. and Zhang, W.},
     TITLE = {Higher Siegel--Weil formula for unitary groups: the non-singular terms},
   JOURNAL = {preprint},
   FJOURNAL={}, 
   VOLUME = {arXiv:2103.11514},
   NUMBER = {},
     YEAR = {},
    PAGES = {},
}

@article{FYZ2,
AUTHOR = {Feng, T. and Yun, Z. and Zhang, W.},
     TITLE = {Higher theta series for unitary groups over function fields},
   JOURNAL = {preprint},
   FJOURNAL={}, 
   VOLUME = {arXiv:2110.07001},
   NUMBER = {},
     YEAR = {},
    PAGES = {},
}

@article {Pap09,
AUTHOR = {Papikian, M.},
     TITLE = {Modular varieties of $\cD$-elliptic sheaves and the Weil--Deligne bound},
   JOURNAL = {J.~Reine Angew.~Math.},
   FJOURNAL={Journal f\"ur die reine und angewandte Mathematik}, 
   VOLUME = {626},
   NUMBER = {},
     YEAR = {2009},
    PAGES = {115--134},
}

\end{document}